\documentclass[12pt]{amsart}
\usepackage{amsfonts,amssymb,amsmath,amscd,amsthm} 
\usepackage{graphicx} 
\usepackage{stmaryrd}
\usepackage[all]{xy}

\usepackage[colorlinks,allcolors=blue]{hyperref} 

\usepackage{longtable}

\newtheorem{thm}{Theorem}
\newtheorem{lem}[thm]{Lemma}
\newtheorem{prop}[thm]{Proposition}

\theoremstyle{definition}

\newtheorem{rmk}[thm]{Remark}

\numberwithin{thm}{section}
\numberwithin{equation}{section}

\newcommand{\Proof}{\noindent {\it Proof}.\ \ }
\newcommand{\Hom}{\operatorname{Hom}}
\newcommand{\sHom}{\mathcal Hom}
\newcommand{\Ext}{\operatorname{Ext}}
\newcommand{\Spec}{\operatorname{Spec}}

\newcommand{\ob}{\operatorname{ob}}
\newcommand{\Hilb}{\operatorname{Hilb}}
\newcommand{\HF}{\operatorname{HF}}
\newcommand{\red}{\operatorname{red}}
\newcommand{\im}{\operatorname{im}}

\newcommand{\Pic}{\operatorname{Pic}}

\newcommand{\coker}{\operatorname{coker}}
\newcommand{\Gr}{\operatorname{Gr}}

\newcommand{\car}{\operatorname{char}}

\newcommand{\codim}{\operatorname{codim}}

\makeatletter
\@namedef{subjclassname@2020}{\textup{2020} Mathematics Subject Classification}
\makeatother

\newcommand{\mapright}[1]{%
  \smash{\mathop{%
      \hbox to 1cm{\rightarrowfill}}\limits^{#1}}}
\newcommand{\mapleft}[1]{%
  \smash{\mathop{%
      \hbox to 1cm{\leftarrowfill}}\limits^{#1}}}

\title[Obstructions to deforming space curves lying on ...]{
  Obstructions to deforming space curves
  lying on a complete intersection of quadrics}
\address{
  Department of Mathematical Sciences,
  Tokai University,
  4-1-1 Kitakaname, Hiratsuka, 
  Kanagawa 259-1292, JAPAN}
\author{Hirokazu Nasu}
\email{nasu@tokai.ac.jp}

\begin{document}

\begin{abstract}
Let $S=S_{2,\ldots,2}\subset\mathbb P^{n+2}$ be a smooth complete intersection of $n$ quadrics containing a line $E$, where $n\geq 3$, and set
$D=(n-1)\mathbf h+E$, where $\mathbf h=[\mathcal O_S(1)]$.
We study the deformations of smooth curves $C\in |mD|$ and their obstructions.
The linear system $|D|$ defines a birational morphism $f:S\to\overline S$ 
that contracts $E$ to a point and realizes an unprojection in the sense of Papadakis--Reid.
When $n=3$, the surface $S=S_{2,2,2}\subset\mathbb P^5$ is a $K3$ surface and $E$ is a $(-2)$-curve.
We prove that the complementary open surface $S\setminus E$ admits first-order infinitesimal deformations in $\mathbb P^5\setminus E$ that are obstructed at the second order.
Using these obstructions, we show that, for every sufficiently large $m$ with $\car k\nmid m$, there exists a {\em generically non-reduced} irreducible component $W_m$ of the Hilbert scheme $\Hilb^{sc}\mathbb P^5$ of smooth connected curves in $\mathbb P^5$.
Its general member $C$ lies on a complete intersection $S_{2,2,2}$ of three quadrics, satisfies $C\sim mD$, and has
$h^0(N_{C/\mathbb P^5})-\dim_{[C]}\Hilb^{sc}\mathbb P^5=3$.
In contrast, when $n>3$, we prove that, for every sufficiently large $m$ with $\operatorname{char} k\nmid m$, the corresponding family of curves in $|mD|$ gives rise to a {\em generically smooth} irreducible component of $\operatorname{Hilb}^{sc}\mathbb P^{n+2}$.
\end{abstract}

\maketitle

\section{Introduction}
\label{sec:intro}

Given a variety $X$, we denote by $\Hilb^{sc} X$
the Hilbert scheme of smooth connected curves in $X$.
Mumford \cite{Mumford} first proved the existence of an irreducible component
of $\Hilb^{sc} \mathbb P^3$,
at whose generic point, $\Hilb^{sc} \mathbb P^3$ is non-reduced.
(We call such a component of the Hilbert scheme
a {\em generically non-reduced component},
or simply, {\em a non-reduced component} below).
Later, many non-reduced components of $\Hilb^{sc} X$ 
were constructed for $X=\mathbb P^3$ 
(cf.~\cite{Gruson-Peskine82,Kleppe87,Ellia87,Nasu1,Kleppe-Ottem15,Dan17,Nasu9}, 
etc., see \cite[\S13]{Hartshorne10} for the historical sketch
about the non-reduced components of the Hilbert schemes after Mumford's first discovery).
However, for $X=\mathbb P^r$ with $r \ge 4$,
relatively few examples have been found so far,
except for \cite{CLM96,CIK24a,CIK24b,Nasu10}.
On the other hand, in \cite{Mukai-Nasu}, Mukai and the author
gave a criterion for a given first-order deformation
$\tilde C \subset X \times \Spec k[t]/(t^2)$ of $C$ in 
a threefold $X$ not to be liftable to the second order. 
This enabled them to generalize Mumford's example for $\Hilb^{sc} \mathbb P^3$
to $\Hilb^{sc} X$ of many uniruled threefolds $X$.
We refer to \cite{Nasu4,Nasu5,Nasu6,Nasu7,Nasu8} 
for subsequent developments in this direction.

In this paper, we study the deformations of curves lying on 
a smooth (global) complete intersection of quadrics (CIQ) 
in the projective space,
assuming that 
the complete intersection is of dimension $2$ (i.e.,~a surface)
and it contains a line. The following is our main theorem.

\begin{thm}
  \label{thm:main1}
  Let $n\ge 3$ and $m \ge 1$ be two integers and 
  $S:=S_{2,\dots,2} \subset \mathbb P:=\mathbb P^{n+2}$ 
  a smooth complete intersection of $n$ quadrics
  containing a line $E \subset \mathbb P$.
  Let $D=(n-1)\mathbf h+E$ be a divisor on $S$,
  where $\mathbf h=[\mathcal O_S(1)] \in \Pic S$, and
  $C$ a general member of a complete linear system $\Lambda:=|mD|$
  on $S$. If $m$ is sufficiently large
  and not a multiple of the characteristic $p$ of the ground field $k$, then
  \begin{enumerate}
    \item \label{item:nonsingularity and dimension of flag}
    $C$ is a smooth connected curve on $S$ and
    the Hilbert-flag scheme $\HF \mathbb P$ of $\mathbb P$
    is nonsingular
    at $(C,S)$ of dimension
    $$
    d_{n,m}:=\dfrac{n(n+2)(n+3)}2+2n+2+\dim \Lambda
    $$
    \item \label{item:conclusion for n=3}
    If $n=3$, then such curves $C$ with variable $S=S_{2,2,2}$ and $E$
    (and the specializations of $C$)
    form an irreducible component $W \subset \Hilb^{sc} \mathbb P$
    of dimension $d_{3,m}$ and $\Hilb^{sc} \mathbb P$ is generically 
    non-reduced along $W$. In fact, 
    $h^0(N_{C'/\mathbb P})=\dim_{[C']} \Hilb^{sc} \mathbb P+3$
    at every general point $[C']$ of $W$.
    \item \label{item:conclusion for n>3}
    If $n> 3$, then such curves $C$ with variable $S=S_{2,\dots,2}$ 
    and $E$ (and the specializations of $C$)
    form a generically smooth irreducible component 
    $W' \subset \Hilb^{sc} \mathbb P$ of dimension $d_{n,m}$.
  \end{enumerate}
\end{thm}
Here the dimension $\lambda_{n,m}:=\dim \Lambda$ is
explicitly given in \S\ref{sec:obstruction} in terms of $n$ and $m$
(see \eqref{eqn:formula for dim|mD|}).
Theorem \ref{thm:main1} provides a new construction of a generically non-reduced component 
of the Hilbert scheme of smooth connected curves in $\mathbb P^5$, 
for which relatively few examples are known.
Our approach is based on the geometry of a line $E$ 
contained in a smooth complete intersection of quadrics and 
on the obstructions to infinitesimal deformations of the complementary open surface $S\setminus E$. 
An interesting feature of our result is the contrast between the cases $n=3$ and $n>3$. 
In the former case, where $S$ is a $K3$ surface, there is a three-dimensional space 
of additional first-order deformations of $C$ that do not arise from deformations 
of the pair $(C,S)$. We prove that every such deformation is obstructed, 
giving rise to a generically non-reduced component.
In the latter case, the cohomology group 
$H^1(N_{S/\mathbb P}\otimes_S \mathcal O_S(-C))$ 
controlling the smoothness of the projection
from the Hilbert-flag scheme vanishes, and the resulting component is generically smooth.
The cases $n=1,2$, which are excluded from Theorem~\ref{thm:main1}, 
are discussed separately in Remark~\ref{rmk:n=2 and n=1}.

It is known that certain classes of closed subschemes, 
such as Cohen–Macaulay subschemes of codimension two, 
complete intersections, 
and Gorenstein subschemes of codimension three,
have well-behaved deformation theories
(cf.~\cite[Chap.~2,~\S\S8-9]{Hartshorne10}).
In these cases, deformations can be constructed 
by lifting the corresponding resolutions of their ideal sheaves.
In particular, we know that if $X \subset \mathbb P^m$
is a complete intersection, then the Hilbert scheme $\Hilb \mathbb P^m$
is nonsingular at the corresponding point $[X]$ 
(see e.g.,~\cite[Corollary~9.5]{Hartshorne10}).
However, this is not necessarily true when we consider
closed subschemes with additional structures.
In fact, the complements of negative curves in complete intersection surfaces 
can have obstructed deformations.
In \cite{Nasu5} the author studied the deformations of
$S^{\circ}:=S \setminus E$ in $X^{\circ}:=X\setminus E$,
assuming that $X$ is a projective $3$-fold,
$S$ is a smooth surface in $X$ and
$E$ is an irreducible curve on $S$ with $E^2 <0$.
In particular, given an infinitesimal deformation
$\beta \in H^0(N_{S/X}\otimes_S \mathcal O_S(E))$ 
with a pole along $E$ of order one,
its primary obstruction $\ob(\beta)$ has been computed
(see Theorem~\ref{thm:open surface in 3-$fold}).
In this paper, we extend this study to the case where $X=\mathbb P^5$, $S=S_{2,2,2}$, 
and $E$ is a line on $S$. The following theorem generalizes the previous results 
for threefolds to a higher-dimensional ambient space and plays a key role 
in proving the non-reducedness assertion in Theorem~\ref{thm:main1}.

\begin{thm}
  \label{thm:main2}
  Let $S_{2,2,2} \subset \mathbb P^5=:\mathbb P$ 
  be a smooth complete intersection of three quadrics 
  containing a line $E \subset \mathbb P$.
  Then the first-order infinitesimal deformation of 
  $S^{\circ} \subset \mathbb P^{\circ}$ induced by any element of
  $H^0(N_{S/\mathbb P}(E))\setminus H^0(N_{S/\mathbb P})$
  does not lift to a deformation over $\Spec k[t]/(t^3)$.
\end{thm}

The organization of this paper is as follows.
In \S\ref{subsec:unprojection}, we generalize the classical standard unprojection of a line on a smooth cubic surface $S_3 \subset \mathbb P^3$ to the case of a line on a complete intersection $S_{2,\dots,2} \subset \mathbb P^{n+2}$ 
of quadrics. We explicitly construct a rational section $\theta$ defining the unprojection, which will play a key role in the obstruction calculations in later sections.
In \S\ref{subsec:vanishing}, \S\ref{subsec:infinitesimal} and \S\ref{subsec:flag and picard}, 
we recall results on vanishing theorems for surfaces, infinitesimal deformations and obstructions, 
and Hilbert-flag schemes and relative Picard schemes, respectively.
In \S\ref{sec:with pole}, we study infinitesimal deformations of the open surface $S\setminus E$ and prove Theorem~\ref{thm:main2}.
Finally, in \S\ref{sec:obstruction}, we apply Theorem~\ref{thm:main2} to the deformation theory of curves on $S$ and prove Theorem~\ref{thm:main1}.

\paragraph{\bf Notation and Conventions}

Throughout the paper, we work over an algebraically closed field $k$ 
of any characteristic $p \ge 0$.
Given a closed subscheme $Y \subset X$,
we denote by $\mathcal I_{Y/X}$ and $N_{Y/X}$ 
the sheaf of ideals defining $Y$ in $X$ and the normal sheaf of $Y$ in $X$,
respectively.
Given a sheaf $\mathcal F$ on $X$, 
$\chi(X,\mathcal F)$ represents the Euler characteristic of $\mathcal F$.
$\Gr(k,V)$ and $\Gr(k,n)$ respectively represent the Grassmannian variety
parametrizing $k$-dimensional linear subspaces of $V$ and $k^n$.
We denote the Hilbert scheme of $X$ by $\Hilb X$ and
its open subscheme parametrizing smooth connected curves in $X$ 
by $\Hilb^{sc} X$.

\section{Preliminaries}
\label{sec:preliminaries}

\subsection{A rational section and an unprojection}
\label{subsec:unprojection}

Unprojection, as developed by Papadakis and Reid \cite{Papadakis-Reid04}, is a systematic construction of Gorenstein rings and varieties by adjoining new generators associated with a suitable subvariety. Geometrically, it can often be viewed as an inverse to projection from that subvariety.
A classical example involves a line on a cubic surface and its relation, via projection and unprojection, to a quartic del Pezzo surface, i.e., a complete intersection of two quadrics \cite[\S2.1]{Reid00}.
In this subsection, we extend this classical example to a line contained in a complete intersection of quadrics of higher codimension.

Our main purpose is to construct a rational section $\theta$ (see \eqref{eqn:section with pole}) 
and certain cohomology classes $t_l$ ($1\leq l\leq n$; cf.~\eqref{eqn:fundamental 1-cocycles}), 
which will be used to detect obstructions to infinitesimal deformations.
The section $\theta$ is used to define the unprojection contracting the line $E$ to a point, 
while the classes $t_l$ describe its cohomological behavior.
In particular, when $n=3$, the coboundary class $\delta(\theta)$ can be expressed in terms of the $t_l$ (Lemma~\ref{lem:coboundary image}), and the linear independence of their restrictions to $E$ (Lemma~\ref{lem:linear independence of coh classes}) plays a crucial role in proving Theorem~\ref{thm:main2}.

Let $n\ge 2$ be an integer and 
$E$ a line in the projective space $\mathbb P^{n+2}$ ($=:\mathbb P$).
We consider a smooth complete intersection surface 
$S=S_{2,\ldots,2}\subset \mathbb P^{n+2}$
of $n$ (hyper)quadrics in $\mathbb P^{n+2}$ containing $E$.
Then by adjunction, we have
$K_S=(n-3)\mathbf h$
and $E^2=-2-K_S\cdot E=1-n$,
where $\mathbf h$ denotes the hyperplane class ($=[\mathcal O_S(1)]$)
of $S$ in $\Pic S$.
Then the restriction $\mathcal O_S(n-1)(E)\big{\vert}_E$ is trivial. 
Moreover, there exists a global section of $\mathcal O_S(n-1)(E)$ 
that does not belong to $H^0(S,\mathcal O_S(n-1))$.
By the exact sequence 
$$
0 \longrightarrow \mathcal O_S(n-1)
\longrightarrow \mathcal O_S(n-1)(E)
\longrightarrow \underbrace{\mathcal O_E(n-1)(E)}_{\simeq \mathcal O_E}
\longrightarrow 0
$$
and $H^1(\mathcal O_S(n-1))=0$ (cf.~Lemma~\ref{lem:vanishing of H^1_*O_S} below),
this section is unique up to scalar multiplication, and 
explicitly described in \cite{Nasu10} in the case $n=2$
by using homogeneous coordinates of 
$\mathbb P^4$ and equations of $S$.
We treat the case $n\ge 3$ in this paper.
First we take homogeneous coordinates 
$x_i$ ($0 \le i\le n+2$) of $\mathbb P^{n+2}$,
so that the first $n+1$ coordinates $x_0,x_1,\dots,x_n$ 
define $E$ in $\mathbb P^{n+2}$.
Then there exist $n$ (linearly independent) quadratic polynomials 
\begin{equation}
  \label{eqn:S}
  q_i := \sum_{j=0}^n a_{ij} x_j \quad (i=1,\dots,n)
\end{equation}
defining $S$ in $\mathbb P^{n+2}$, where $a_{ij}$
are linear polynomials on $\mathbb P^{n+2}$ ($j=0,1,\dots,n$).
We consider an $n\times (n+1)$-matrix $A$ defined by
\begin{equation}
  \label{eqn:matrix A}
  A:=
  \begin{pmatrix}
    a_{10} & a_{11} & \dots & a_{1n} \\
    a_{20} & a_{21} & \dots & a_{2n} \\
    \vdots & \vdots & &  \vdots \\
    a_{n0} & a_{n1} & \dots & a_{nn}
  \end{pmatrix}.
\end{equation}
For every $0 \le j \le n$,
we denote the column of $A$ with index $j$  by $\mathbf a_j$
and we write
$$
A=(\mathbf a_0,\dots,\mathbf a_n).
$$
Let $A_j$ denote the matrix obtained from $A$ 
by removing its $j$-th column $\mathbf a_j$.
Then the $n$-minors of $A$ are, up to sign, 
the determinants of $A_i$, more precisely, for $j=0,\dots,n$,
the $j$-th minor $\Delta_j$ of $A$ is given by the homogeneous polynomial
$$
\Delta_j=(-1)^j \det A_j
$$ 
of degree $n$.
Similarly, for $j_1,j_2 \in [0,n]$ with $j_1 \ne j_2$, 
we denote by $A_{j_1j_2}$
the $n \times (n-1)$-matrix 
obtained from $A$ by removing $\mathbf a_{j_1}$ and $\mathbf a_{j_2}$.
Here and later the symbol $[i,j]$ denotes 
the set of consecutive integers from $i$ to $j$.
Then we obtain the following lemma.
\begin{lem}
  \label{lem:equality of ratio}
  If $i<j$ then the equations
  \begin{equation}
    \label{eqn:equality of ratio}
    x_i \Delta_j - x_j \Delta_i
    =(-1)^{i+j} \det (\mathbf q | A_{ij})
  \end{equation}
  on $\mathbb P^{n+2}$ hold, where
  $\mathbf q={}^t (q_1,\dots,q_n)$.
  Here and later, ${}^t M$ denotes the transpose of a matrix $M$.
\end{lem}
\Proof
By the linearity and alternating property of determinants,
\begin{align*}
  x_i \Delta_j
  &=x_i(-1)^j\det(\mathbf a_0, \dots, \mathbf a_i, \dots, \widehat{\mathbf a_j}, \dots \mathbf a_n)\\
  &=(-1)^{i+j}\det(x_i\mathbf a_i | A_{ij})
\end{align*}
and similarly $x_j\Delta_i=(-1)^{i+j-1}\det (x_j \mathbf a_j| A_{ij})$.
Thus the left hand side of \eqref{eqn:equality of ratio} is equal to
$(-1)^{i+j}\det (x_i\mathbf a_i+x_j \mathbf a_j | A_{ij})$, where
$$
x_i \mathbf a_i + x_j \mathbf a_j
=\mathbf q- \sum_{k \ne i,j} x_k \mathbf a_k
$$
by \eqref{eqn:S}.
Thus the alternating property of determinants yields \eqref{eqn:equality of ratio}.
\qed

\medskip

By the lemma, there exists a rational section $\theta$ of $\mathcal O_S(n-1)$
defined by
\begin{equation}
  \label{eqn:section with pole}
  \theta = \dfrac{\Delta_0}{x_0} 
  = \dfrac{\Delta_1}{x_1} = \dots = \dfrac{\Delta_n}{x_n}.
\end{equation}
Let $D(f)$ denote the standard affine open subset of $\mathbb P^{n+2}$
(or $S$) associated to a section $f$ of 
an invertible sheaf on $\mathbb P^{n+2}$ (or $S$).
By the smoothness of $S$, 
the matrix $A(p)$ has rank $n$ at every point $p\in E$, 
so that $\Delta_j(p)\ne 0$ for some $j$. 
Since $D(x_j)$ ($0\le j\le n$) cover $\mathbb P^{n+2}\setminus E$, 
the $2(n+1)$ open subsets $D(x_j)$ and $D(\Delta_j)$
cover $\mathbb P^{n+2}$ (and hence $S$).
Here and later, we abuse notation and denote by the same symbols
$x_i$ (resp.~$\Delta_j$) the restrictions to $S$ of
the global sections of $\mathcal O_{\mathbb P^{n+2}}(1)$ 
(resp.~$\mathcal O_{\mathbb P^{n+2}}(n)$).
We denote the open covering 
$\left\{D(x_j),D(\Delta_j) \bigm| 0 \le j \le n \right\}$ of $\mathbb P^{n+2}$
(or $S$) by $\mathfrak U_1$.

\begin{prop}
  \label{prop:global section}
  \begin{enumerate}
    \item $\theta$ defines a global section of $\mathcal O_S(n-1)(E)$,
    but $\theta$ does not belong to $H^0(S,\mathcal O_S(n-1))$.
  
    \item The restriction $\theta\big{\vert}_E$ of $\theta$ to $E$
    gives a trivialization
    \begin{equation}
      \label{isom:trivialization}
      \mathcal O_E \mapright{\sim} \mathcal O_S(n-1)(E)\big{\vert}_E, 
      \qquad
      \mu \longleftrightarrow \mu (\theta\big{\vert}_E).
    \end{equation}
  \end{enumerate}
\end{prop}
\Proof
Since $x_0,\dots,x_n$ form a basis of 
$H^0(\mathbb P^{n+2},\mathcal I_{E/\mathbb P^{n+2}}(1))$,
their restrictions to $S$ form a basis of 
$H^0(\mathcal I_{E/S}(1))\simeq H^0(\mathcal O_S(1)(-E))\simeq k^{n+1}$.
The complete linear system $|\mathcal O_S(1)(-E)|$ ($\simeq |\mathbf h-E|$)
defines a morphism $\Phi': S\setminus E \rightarrow \mathbb P^n$, sending
$p \mapsto (x_0(p):\dots :x_n(p))$, 
corresponding to the projection of $S$ from $E$.
It follows from \eqref{eqn:section with pole} that
the morphism extends to $\Phi:=\Phi_{|\mathbf h-E|}: S \rightarrow \mathbb P^n$
 (Here $\Phi(p)=(\Delta_0(p):\dots :\Delta_n(p))$ if $p \in E$).
Let $x_j'$ be the global section of $\mathcal O_S(\mathbf h-E)$
corresponding to $x_j$ ($0 \le j\le n$). Then 
\begin{equation}
  \label{eqn:x_i and x_i'}
  D(x_j)\cap S=D(x_j')\setminus E
\end{equation}
for all $0 \le j\le n$.
We see that $\mathfrak U_2:=\left\{D(x_j')\bigm| 0 \le j\le n\right\}$ 
covers whole $S$, because $|\mathbf h-E|$ is base point free.
Thus, by \eqref{eqn:section with pole}, 
$\theta$ defines a global section of $\mathcal O_S(n-1)(E)$.
(2) immediately follows from (1).
\qed

\medskip

The following lemma shows that the linear system $|\mathcal O_S(n-1)(E)|$ on $S$
defines an {\em unprojection morphism} of $E$ on $S$.

\begin{lem}
  \label{lem:unprojection}
  The complete linear system $|\mathcal O_S(n-1)(E)|$ on $S$ defines 
  a proper birational morphism
  $$
  f: S \rightarrow \mathbb P(H^0(S,\mathcal O_S(n-1)(E))^{\vee})=\mathbb P^N,
  $$
  where $N=h^0(\mathcal O_S(n-1))$.
  This morphism $f$ contracts $E$ to a point $q$
  and the restriction 
  $f\big{\vert}_{S\setminus E}$ of $f$ has the inverse morphism
  defined by the projection of $f(S)$ from $q$.
\end{lem}
\Proof Let $v_{n-1}: S \hookrightarrow \mathbb P^{N-1}$
be the Veronese embedding of $S$  
into $\mathbb P^{N-1}=\mathbb P(H^0(\mathcal O_S(n-1))^{\vee})$
of degree $n-1$, and $y_0,\dots,y_{N-1}$ the global sections of
$\mathcal O_S(n-1)=v_{n-1}^*(\mathcal O_{\mathbb P^{N-1}}(1))$.
Then $f: S \rightarrow \mathbb P^N$ 
is defined by $f(p)=(y_0(p),\dots,y_{N-1}(p),\theta(p))$.

For every $p\in E$, the sections $y_0,\ldots,y_{N-1}$,
regarded as sections of $\mathcal O_S((n-1)\mathbf h+E)$, vanish at $p$,
whereas $\theta|_E$ is nowhere vanishing by
Proposition~\ref{prop:global section}. Hence
$f(E)=\{q\}$, where $q:=(0:\cdots:0:1)$. Moreover, $f^{-1}(q)=E$.
Let $\pi_q$ denote the projection from $q$ onto
$\mathbb P^{N-1}$. On $S\setminus E$ we have $\pi_q\circ f=v_{n-1}$.
Since $v_{n-1}$ is an embedding, the morphism
$v_{n-1}^{-1}\circ\pi_q$ is the inverse of
$f\big{\vert}_{S\setminus E}: S\setminus E\longrightarrow f(S)\setminus\{q\}$.
Consequently, $f$ is birational and its exceptional locus is
exactly $E$.
\qed

\medskip

We next define a series of important $1$-cochains 
of $\mathcal O_S(K_S+2E)$ on $S$.
We follow the notation in Perrin's book~\cite[Chap.~X,~\S2~,pp.182--183]{Perrin08}.
Let $A_{i;j_1,j_2}$ denote the $(n-1)\times (n-1)$-matrix obtained from $A$ 
(cf.~\eqref{eqn:matrix A})
by deleting the row of index $i$ 
and the two columns of indices $j_1,j_2$ for $j_1 \ne j_2$.
The corresponding $(n-1)$-minor $\Delta_{i;j_1,j_2}$ of $A$ is given by
the homogeneous polynomial 
\begin{equation}
  \label{eqn:(n-1)-minors}
  \Delta_{i;j_1,j_2}:=
  \begin{cases}
    (-1)^{i+j_1+j_2} \det A_{i;j_1,j_2} & \mbox{if $j_1 < j_2$} \\
    (-1)^{i+j_1+j_2+1} \det A_{i;j_1,j_2} & \mbox{if $j_1 > j_2$} 
  \end{cases}
\end{equation}
of degree $n-1$.
We extend this definition to the case $j_1=j_2$ by setting $\Delta_{i;j,j}=0$.
The formula for expanding a determinant along a row applied to $A_j$ gives us
\begin{equation}
  \label{eqn:key formula}
  \sum_{k\ne j} a_{ik} \Delta_{l;j,k}=
  \sum_{k=0}^n a_{ik} \Delta_{l;j,k} = \delta_{il} \Delta_j
\end{equation}
for all $i,l \in [1,n]$ and $j \in [0,n]$,
where $\delta_{il}$ represents Kronecker's delta, i.e.,
$\delta_{il}=0$ if $i\ne l$ and $\delta_{ii}=1$
(See \cite[p.183,~l.3]{Perrin08} for similar formulas).

For each $l\in [1,n]$, 
we define the corresponding $1$-cochains $\hat t_l$ of 
$\mathcal O_S(K_S+2E)$ by
\begin{equation}
  \label{eqn:fundamental 1-cocycles}
  (\hat t_l)_{ij}:=
    \dfrac{\Delta_{l;i,j}}{x_ix_j}
  \in 
  C^1(\mathfrak U_2,\mathcal O_S(K_S+2E)).
\end{equation}
Here, the symbol $(\hat t_l)_{ij}$ for $i,j\in [0,n]$ with $i\ne j$
stands for the data of $\hat t_l$ over $D(x_i'x_j')$,
which is a local section of $\mathcal O_S(K_S+2E)\simeq \mathcal O_S(n-3)(2E)$.
It is easy to check that each of 
the $1$-cochains $\hat t_l$ ($1 \le l \le n$)
satisfies the cocycle conditions,
i.e.,
\begin{equation}
  \label{eqn:cocycle condition1}
  \dfrac{\Delta_{l;i,j}}{x_ix_j}+
  \dfrac{\Delta_{l;j,k}}{x_jx_k}+
  \dfrac{\Delta_{l;k,i}}{x_kx_i}=0
  \quad
  \mbox{for all $i,j,k$}
\end{equation}
by a similar computation given in the proof of 
Lemma~\ref{lem:equality of ratio}.
Thus $\hat t_l$ represents an 
element $t_l$ in $H^1(S,\mathcal O_S(K_S+2E))$.
In what follows, by abuse of notation,
we denote a cohomology class and a cochain representing it by the same symbol.
Thus we write 
$t_l:=\left( \Delta_{l;i,j}/x_ix_j \right)_{ij} \in H^1(S,\mathcal O_S(K_S+2E))$.
We see that the restriction of $\mathcal O_S(K_S+2E)$ to $E$
is isomorphic to $\mathcal O_E(-n-1)$.
The following lemma is a generalization of \cite[Lemma~4.2]{Nasu10} and
will play a key role in our proof of Theorems~\ref{thm:main1} and \ref{thm:main2} 
later.

\begin{lem}
  \label{lem:linear independence of coh classes}
  Under the trivialization \eqref{isom:trivialization}, 
  $$
  t_l\big{\vert}_E
  =\left(
    \dfrac{\Delta_{l;i,j}}{\Delta_i\Delta_j}
  \right)_{ij}
  \in 
  H^1(\mathfrak U_2,\mathcal O_E(-n-1))
  $$
  for all $l \in [1,n]$.
  Moreover, $t_1\big{\vert}_E,\dots,t_n\big{\vert}_E$ are linearly independent, 
  and hence they form a basis of
  $H^1(E,\mathcal O_E(-n-1))\simeq k^n$.
\end{lem}
\Proof 
We first note that 
the data $(\Delta_{l;i,j}/\Delta_i\Delta_j)_{ij}$ ($0\le i,j\le n$) satisfy
the cocycle condition, i.e.,
\begin{equation}
  \label{eqn:cocycle condition2}
  \dfrac{\Delta_{l;i,j}}{\Delta_i\Delta_j}+
  \dfrac{\Delta_{l;j,k}}{\Delta_j\Delta_k}+
  \dfrac{\Delta_{l;k,i}}{\Delta_k\Delta_i}=0
  \quad
  \mbox{for all $i,j,k$},
\end{equation}
which can be obtained from
the multiplication of \eqref{eqn:cocycle condition1} by $\theta^{-2}$.
There exists a natural exact sequence 
\begin{equation}
  \label{ses:normal bundles}
  \begin{CD}
    0 @>>> \underbrace{N_{E/S}}_{\simeq \mathcal O_E(1-n)} 
    @>{\iota}>> 
    \underbrace{N_{E/\mathbb P^{n+2}}}_{\simeq \mathcal O_E(1)^{\oplus n+1}}
    @>{\pi_{E/S}}>> 
    \underbrace{N_{S/\mathbb P^{n+2}}\big{\vert}_E}_{\simeq \mathcal O_E(2)^{\oplus n}}
    @>>> 0
  \end{CD}
\end{equation}
on $E \simeq \mathbb P^1$. Taking the dual bases of 
$x_j$ ($0 \le j \le n$) and $q_i$ ($1 \le i \le n$), we see that
the sheaf homomorphisms $\pi_{E/S}$ and $\iota$
are represented by the matrix $A$ in \eqref{eqn:matrix A} and 
${}^t (\Delta_0, \dots, \Delta_n)$, respectively. 
The exact sequence obtained by tensoring $\eqref{ses:normal bundles}$ 
with $\mathcal O_E(-2)$ induces a coboundary map
$$
\partial_{E/S}: H^0(E,
\underbrace{N_{S/\mathbb P^{n+2}}\big{\vert}_E\otimes \mathcal O_E(-2)}_{\simeq \mathcal O_E^{\oplus n}}) \rightarrow H^1(E,\underbrace{N_{E/S}\otimes \mathcal O_E(-2)}_{\mathcal O_E(-n-1)}),
$$
which is clearly an isomorphism.
We fix $l \in [1,n]$ and compute the image of the $l$-th standard basis
$\mathbf e_l={}^t (0,\dots,\overset{l}{1},\dots,0)$
of $H^0(\mathcal O_E^{\oplus n})$.
We put
$$
s_j := \Delta_j^{-1}
\begin{pmatrix}
  \Delta_{l;j,0}
  \\ \vdots \\
  \Delta_{l;j,n}
\end{pmatrix}
$$
for $0 \le j\le n$. 
A direct computation shows that $\pi_{E/S}(s_j)=\mathbf e_l$,
because the $i$-th entry of $\pi_{E/S}(s_j)
=\Delta_j^{-1} A {}^t\left(
  \Delta_{l;j,0},\dots,\Delta_{l;j,n}
  \right)$ is equal to $\delta_{il}$
by \eqref{eqn:key formula}.
Thus $s_j$ is a local lift of $\mathbf e_l$ by $\pi_{E/S}$ over $D(x_j')$.
Then the coboundary of the $0$-cochain $(s_j)_{0 \le j \le n}$ is equal to
$$
s_j\big{\vert}_{D(x_i'x_j')} - s_i\big{\vert}_{D(x_i',x_j')}
= \dfrac{1}{\Delta_j}
\begin{pmatrix}
  \Delta_{l;j,0}
  \\ \vdots \\
  \Delta_{l;j,n}
\end{pmatrix}
- \dfrac{1}{\Delta_i}
\begin{pmatrix}
  \Delta_{l;i,0}
  \\ \vdots \\
  \Delta_{l;i,n}
\end{pmatrix}
$$
in $\Gamma(D(x_i'x_j'), \mathcal O_E(-1)^{\oplus n+1})$.
The $k$-th entry of the last equation is equal to 
$$
\dfrac{\Delta_{l;jk}}{\Delta_j}-\dfrac{\Delta_{l;ik}}{\Delta_i}
=\dfrac{\Delta_{l;jk}}{\Delta_j}+\dfrac{\Delta_{l;ki}}{\Delta_i}
=-\Delta_k \left(\dfrac{\Delta_{l;ij}}{\Delta_i \Delta_j} \right)
$$
by \eqref{eqn:cocycle condition2}.
This coincides with the $k$-th entry of 
$\iota (-t_l\big{\vert}_E)$.
Therefore we have proved that the coboundary image of $\mathbf e_l$
is equal to $-t_l\big{\vert}_E$ in $H^1(\mathcal O_E(-n-1))$ 
for any $1 \le l\le n$.
Since $\mathbf e_1,\dots,\mathbf e_n$ are clearly linearly independent, 
so are $t_1\big{\vert}_E,\dots,t_n\big{\vert}_E$,
and the proof has been completed.
\qed

\subsection{Vanishing theorems on CIQ}
\label{subsec:vanishing}

Let $n \ge 2$ be an integer, $S$ a (smooth) complete intersection 
$S_{2,...,2} \subset \mathbb P^{n+2}$ of $n$ quadrics
containing a line $E$ on $\mathbb P^{n+2}$.

\begin{lem}\label{lem:higher-direct-image}
Let $f\colon S\longrightarrow \overline S$
be the birational morphism defined by $|D|$, 
where
$D=(n-1)\mathbf h+E$ (cf.~Lemma~\ref{lem:unprojection}),
and $\mathcal F$ an invertible sheaf on $S$.
If $e:=\deg \mathcal F\big{\vert}_E \ge -1$, then
$$
R^1f_*\mathcal F=0.
$$
\end{lem}
\Proof
By Lemma~\ref{lem:unprojection}, the exceptional locus of $f$ is
exactly $E$. Let $q=f(E)$. Since $f$ is an isomorphism over
$\overline S\setminus\{q\}$,
the sheaf $R^1f_*\mathcal F$ is supported at $q$.
The infinitesimal neighborhoods of the fiber $f^{-1}(q)$
and the thickenings $rE$ form cofinal systems.
Hence, by the theorem on formal functions, we obtain
$$
\widehat{\bigl(R^1f_*\mathcal F\bigr)}_q
\simeq
\varprojlim_r
H^1\bigl(rE,\mathcal F\big{\vert}_{rE}\bigr).
$$
We show that $H^1\bigl(rE,\mathcal F\big{\vert}_{rE}\bigr)=0$
for every $r\geq1$. For $r=1$, we have
$
\mathcal F\big{\vert}_E \simeq \mathcal O_{\mathbb P^1}(e)$,
and hence
$
H^1\bigl(E,\mathcal F\big{\vert}_E\bigr)=0$.
For $r\geq2$, the natural filtration of $rE$ gives an exact sequence
$$
0\longrightarrow
\mathcal F\otimes_S \mathcal O_E\bigl(-(r-1)E\bigr)
\longrightarrow
\mathcal F\big{\vert}_{rE}
\longrightarrow
\mathcal F\big{\vert}_{(r-1)E}
\longrightarrow0
$$
for every $r\geq2$. Since $E^2=1-n$, we have
$$
\mathcal F\otimes_S \mathcal O_E\bigl(-(r-1)E\bigr)
\simeq
\mathcal O_{\mathbb P^1}
\bigl(e+(r-1)(n-1)\bigr),
$$
where $e+(r-1)(n-1)\geq -1$ by $e \ge -1$ and $r\ge 1$.
Therefore, $H^1\bigl(\mathcal F(-(r-1)E)\big{\vert}_E \bigr)=0$.
By induction on $r$, we obtain $H^1\bigl(rE,\mathcal F|_{rE}\bigr)=0$
for every $r\geq1$ and hence $R^1f_*\mathcal F=0$.
\qed

\begin{lem}
  \label{lem:vanishing-D}
  Set $D=(n-1)\mathbf h+E$. Then, for all sufficiently large
  integers $m$, we have
  \begin{enumerate}
    \item $H^1(S,\mathcal O_S(mD))=0$,
    \item $H^1(S,\mathcal O_S(-mD))=0$,
    \item $H^1(S,\mathcal O_S(mD-2E))=0$ and
    \item $H^1(S,\mathcal O_S(2\mathbf h-mD))=0$ if $n>3$.
  \end{enumerate}
\end{lem}
\Proof
By Lemma~\ref{lem:unprojection}, there exists an ample Cartier
divisor $A$ on $\overline S$ such that
$
D=f^*A.
$
Let $\mathcal F$ be one of the coherent sheaves
$\mathcal O_S$, $\omega_S$ ($\simeq \mathcal O_S(n-3)$) and $\mathcal O_S(-2E)$.
It follows from Lemma~\ref{lem:higher-direct-image} that $R^1f_* \mathcal F=0$.
Then the projection formula shows that
$$
R^1f_*
\bigl(\mathcal F\otimes\mathcal O_S(mD)\bigr)
\simeq
R^1f_*\mathcal F\otimes
\mathcal O_{\overline S}(mA)
=0.
$$
It follows from the Leray spectral sequence that
$$
H^1\bigl(
S,\mathcal F\otimes\mathcal O_S(mD)
\bigr)
\simeq
H^1\bigl(
\overline S,
f_*\mathcal F\otimes\mathcal O_{\overline S}(mA)
\bigr).
$$
Since $f_*\mathcal F$ is coherent and $A$ is ample, the right-hand
side vanishes for $m\gg0$ by Serre vanishing.
By taking $\mathcal F=\mathcal O_S$ and
$\mathcal F=\mathcal O_S(-2E)$, we have
$H^1(S,\mathcal O_S(mD))=H^1(S,\mathcal O_S(mD-2E))=0$ for $m\gg0$.
Similarly, the Serre duality shows
$H^1(S,\mathcal O_S(-mD))=H^1(S,\omega_S(mD))^{\vee}=0$ for $m\gg0$.

Finally, assume that $n>3$ and set
$\mathcal G:=\mathcal O_S((n-5)\mathbf h)$.
Then by Lemma~\ref{lem:higher-direct-image}, we have $R^1f_*\mathcal G=0$.
Therefore, the same argument using 
the projection formula and Leray spectral sequence shows
$$
H^1\bigl(S,\mathcal O_S(mD+(n-5)\mathbf h)\bigr)
\simeq
H^1\bigl(\overline S,f_*\mathcal G\otimes \mathcal O_{\overline S}(mA)\bigr),
$$
where the last cohomology groups vanish for $m\gg0$.
Since $K_S=(n-3)\mathbf h$, we have the final assertion by the Serre duality.
\qed

\begin{lem}
\label{lem:vanishing of H^1_*O_S}
Let $S$ be a smooth complete intersection surface 
$S_{2,\ldots,2}\subset \mathbb P^{n+2}$.
Then 
$H^1(S,\mathcal O_S(m))=0$
for every integer $m$.
\end{lem}
\Proof
It follows from the standard exact sequence
$[0 \rightarrow \mathcal I_S 
\rightarrow \mathcal O_{\mathbb P}
\rightarrow \mathcal O_S
\rightarrow 0]\otimes_{\mathbb P} \mathcal O_{\mathbb P}(m)
$
that $H^1(\mathcal O_S(m\mathbf h))\simeq H^2(\mathcal I_S(m))$.
Consider the Koszul resolution 
$$
0 \rightarrow \underbrace{\mathcal O_{\mathbb P}(-2n)}_{=\mathcal L_{n-1}}
\rightarrow \underbrace{\mathcal O_{\mathbb P}(-2n+2)^{\oplus n}}_{=\mathcal L_{n-2}}
\rightarrow \dots
\rightarrow \underbrace{\mathcal O_{\mathbb P}(-4)^{\oplus \binom{n}{2}}}_{=\mathcal L_1}
\rightarrow \underbrace{\mathcal O_{\mathbb P}(-2)^{\oplus n}}_{=\mathcal L_0}
\rightarrow \mathcal I_S
\rightarrow 0
$$
of the ideal sheaf $\mathcal I_S$ of $S$ in $\mathbb P^{n+2}$.
We put $\mathcal K_i:=\im u_{i+1}: \mathcal L_{i+1} \rightarrow \mathcal L_i$
($0 \le i\le n-2$).
Then we see that
$$
H^2(\mathbb P^{n+2}, \mathcal I_S(m)) \simeq H^3(\mathbb P^{n+2}, \mathcal K_0(m))
\simeq 
\dots
\simeq H^{n+1}(\mathbb P^{n+2}, \mathcal K_{n-2}(m)).
$$
Since $\mathcal K_{n-2} = \mathcal L_{n-1}$,
we have $H^{n+1}(\mathcal K_{n-2}(m))\simeq 
H^{n+1}(\mathcal O_{\mathbb P^{n+2}}(m-2n))=0$
by $0 < n+1 < n+2$.
Thus we have proved the lemma. \qed

\subsection{Infinitesimal deformations and obstructions}
\label{subsec:infinitesimal}

In this subsection, we recall some basic facts about 
infinitesimal deformations and obstructions of embedded varieties, 
or more generally, of closed subschemes.
Let $X$ be a scheme over $k$, and $C$ a closed subscheme of $X$.
Then the first-order infinitesimal deformations
$\tilde C \subset X \times_k k[t]/(t^2)$ of $C$ in $X$ 
are in one-to-one correspondence with
global sections $\alpha$ of the normal sheaf $N_{C/X}$ of $C$.
For simplicity, in what follows,
we assume that $C$ is a locally complete intersection in $X$.
Then the normal sheaf $N_{C/X}$ of $C$ in $X$ is a locally free sheaf on $C$
and by \cite[Theorem~4.3.5]{Sernesi},
for each $\alpha$, there exists a unique element $\ob(\alpha)$
of $H^1(N_{C/X})$ ($\hookrightarrow \Ext^1(\mathcal I_C,\mathcal O_C)$)
such that $\ob(\alpha)=0$ if and only if
$\tilde C$ extends to a deformation 
${\tilde {\tilde C}}$ of $C$ over $k[t]/(t^3)$. 
Here $\ob(\alpha)$ is called the {\em primary obstruction} of $\alpha$,
and $\ob(\alpha)\ne 0$ implies that 
the Hilbert scheme of $X$ is singular at $[C]$
by infinitesimal lifting property of smoothness.
The natural surjection
$\mathcal I_C \twoheadrightarrow 
\mathcal I_C/\mathcal I_C^2 = N_{C/X}^{\vee}$ induces the following 
two isomorphisms
$$
H^0(N_{C/X})\simeq \Hom_C(N_{C/X}^{\vee},\mathcal O_C) 
\quad
\mbox{and}
\quad
H^1(N_{C/X})\simeq \Ext^1_C(N_{C/X}^{\vee},\mathcal O_C)
$$
of cohomology groups on $C$.
By the first isomorphism, we identify the global section $\alpha$ of $N_{C/X}$ 
with the corresponding
global $\mathcal O_C$-linear homomorphism from $N_{C/X}^{\vee}$ to $\mathcal O_C$.
Similarly, by the second isomorphism,
$\ob(\alpha)$ is identified with an element in
$\Ext^1_C(N_{C/X}^{\vee},\mathcal O_C)$, and
under the identifications, $\ob(\alpha)$ can be expressed by
the cup product formula
$$
\ob(\alpha)=\alpha \cup \mathbf e \cup \alpha,
$$
where $\mathbf e \in \Ext^1_X(\mathcal O_C,N_{C/X}^{\vee})$ is 
the extension class of the exact sequence
\begin{equation}
  \label{ses:conormal}
  0 \rightarrow N_{C/X}^{\vee} \rightarrow \mathcal O_X/\mathcal I_C^2
  \rightarrow \mathcal O_C \rightarrow 0
\end{equation}
on $X$ (cf.~\cite[Theorem~2.1]{Nasu5} and \cite{Nasu10}). 
Primary obstructions of the first-order deformations of projective varieties
can be computed by using the following lemma.
For related results on detecting obstructions to deformations of space curves 
$C \subset \mathbb P^3$ via cup products, see Fl{\o}ystad \cite{Floystad93} 
and Kleppe \cite[Proposition 3.8]{Kleppe06}.

\begin{lem}
  \label{lem:primary obstruction}
  Let $C \subset \mathbb P^n$ be a locally complete intersection 
  closed subscheme.
  Then for every integer $l$,
  $\ob(\alpha)=\alpha \cup \mathbf e \cup \alpha
  \in \Ext^1_C(N_{C/\mathbb P}^{\vee},\mathcal O_C)$ induces 
  a sequence 
  \begin{equation}
    \label{map:primary obstruction}
    \begin{CD}
      H^0(N_{C/\mathbb P^n}^{\vee}(l))
      @>{\cup \alpha}>>
      H^0(\mathcal O_C(l))
      @>{\cup \mathbf e}>>
      H^1(N_{C/\mathbb P^n}^{\vee}(l))
      @>{\cup \mathbf \alpha}>>
      H^1(\mathcal O_C(l))
    \end{CD}
  \end{equation}
  of (degree-preserving) cup product maps on $\mathbb P^n$ (or $C$),
  where $\cup \mathbf e$ is the first coboundary map of
  $\eqref{ses:conormal}\otimes \mathcal O_{\mathbb P^n}(l)$.  
  In particular, if the composite of the cup product maps is nonzero, 
  then $\ob(\alpha)\ne 0$.
\end{lem}

We also consider deformations of closed subschemes relative to a fixed intermediate closed subscheme.
Let $S \subset X$ be a closed subscheme containing $C$ 
and such that the two embeddings
$C \hookrightarrow S$ and $S \hookrightarrow X$ are regular.
We consider the images of $\alpha$ and $\ob(\alpha)$ 
in $H^i(N_{S/X}\big{\vert}_C)$ ($i=0,1$) by the natural projection
$\pi_{C/S}: N_{C/X} \rightarrow N_{S/X}\big{\vert}_C$, respectively.
We denote the two images by $\pi_{C/S}(\alpha)$ and $\ob_S(\alpha)$,
and call them the {\em exterior components} of $\alpha$ and $\ob(\alpha)$,
respectively. The following lemma is a generalization
of {\cite[Lemma~2.4]{Mukai-Nasu}} in the case $X=\mathbb P^n$.

\begin{lem}\label{lem:reduction}
  Let $C \subset \mathbb P^n$ ($=:\mathbb P$)
  be a locally complete intersection closed subscheme
  contained in a global complete intersection subscheme $S$ in $\mathbb P^n$.
  Let $\pi_{C/S}(\alpha)$ and $\ob_S(\alpha)$ be 
  the exterior components of $\alpha$ and $\ob(\alpha)$, respectively.
  If there exists a global section $\beta$ of $N_{S/\mathbb P}$
  such that $\beta\big{\vert}_C=\pi_{C/S}(\alpha)$, then we have
  $$
  \ob_S(\alpha) = \ob(\beta)\big{\vert}_C
  $$
  in $H^1(C,N_{S/\mathbb P}\big{\vert}_C)$.
\end{lem}
\Proof
Since $S \subset \mathbb P$ is a global complete intersection,
there exist a hypersurface $Q_i$ of degree $d_i$ ($1 \le i \le n$)
such that $S=Q_1 \cap \dots \cap Q_n$, i.e., $S=S_{d_1,\dots,d_n}$,
where $n=\codim_{\mathbb P} S$.
Let $\mathbf e_C \in \Ext^1(\mathcal O_C,N_{C/\mathbb P})$ 
and $\mathbf e_S \in \Ext^1(\mathcal O_S,N_{S/\mathbb P})$ 
be the extension classes corresponding to
the exact sequences \eqref{ses:conormal}
for $C \subset \mathbb P$ and $S \subset \mathbb P$, respectively.
There exists a commutative diagram
$$
\begin{CD}
  0 @>>> N_{S/\mathbb P}^{\vee} @>>> \mathcal O_\mathbb P/\mathcal I_S^2
  @>>> \mathcal O_S @>>> 0 \\
  @. @VVV @VVV @V{|_C}VV @. \\
  0 @>>> N_{C/\mathbb P}^{\vee} @>>> \mathcal O_\mathbb P/\mathcal I_C^2
  @>>> \mathcal O_C @>>> 0
\end{CD}
$$
of $\mathcal O_{\mathbb P}$-modules.
We note that the left vertical map from $N_{S/\mathbb P}^{\vee}$ to $N_{C/\mathbb P}^{\vee}$
factors through 
$\pi_{C/S}^{\vee}: N_{S/\mathbb P}^{\vee}\big{\vert}_C 
\hookrightarrow N_{C/\mathbb P}^{\vee}$,
since $N_{C/\mathbb P}^{\vee}$ is a $\mathcal O_C$-module.
Then the diagram induces a commutative diagram
$$
\begin{CD}
  H^0(S,\mathcal O_S(l)) @>{\cup \mathbf e_S}>> H^1(S,N_{S/\mathbb P}^{\vee}(l)) \\
  @V{|_C}VV @VV{(\pi_{C/S}^{\vee}\otimes \mathcal O_C(l)) \circ |_C}V \\
  H^0(C,\mathcal O_C(l)) @>{\cup \mathbf e_C}>> H^1(C,N_{C/\mathbb P}^{\vee}(l)) \\
\end{CD}
$$
for any integer $l$. 
We note that $N_{S/\mathbb P}\simeq \oplus_{i=1}^n \mathcal O_S(d_i)$.
Taking the direct sum of the above diagram
over $l\in \left\{d_1,\dots,d_n \right\}$, 
we obtain a commutative diagram
$$
\begin{CD}
  H^0(S,N_{S/\mathbb P}) @>{\cup \mathbf e_S}>> 
H^1(S,N_{S/\mathbb P}^{\vee}\otimes N_{S/\mathbb P}) \\
  @V{|_C}VV @VV{(\pi_{C/S}^{\vee}\otimes N_{S/\mathbb P}) \circ |_C}V \\
  H^0(C,N_{S/\mathbb P}\big{\vert}_C) @>{\cup \mathbf e_C}>> 
  H^1(C,N_{C/\mathbb P}^{\vee}\otimes N_{S/\mathbb P}\big{\vert}_C).
\end{CD}
$$
By the commutativity, we have
$$
\pi_{C/S}^{\vee}((\beta \cup \mathbf e_S)\big{\vert}_C)
=\beta\big{\vert}_C \cup \mathbf e_C
= \pi_{C/S}(\alpha) \cup \mathbf e_C.
$$
There exists another commutative diagram
$$
\begin{CD}
  H^1(C,N_{C/\mathbb P}^{\vee} \otimes N_{S/\mathbb P}) 
   @.\times @.  H^0(C,N_{C/\mathbb P}) @>{\cup}>> H^1(C,N_{S/\mathbb P}\big{\vert}_C) \\
  @A{\pi_{C/S}^{\vee}}AA @.  @V{\pi_{C/S}}VV  \Vert \\
  H^1(C,N_{S/\mathbb P}^{\vee}\big{\vert}_C \otimes N_{S/\mathbb P}) 
  @.  \times @.  H^0(C,N_{S/\mathbb P}\big{\vert}_C) @>{\cup}>> H^1(C,N_{S/\mathbb P}\big{\vert}_C)
\end{CD},
$$
which shows that $\pi_{C/S}^{\vee}(A)\cup B=A \cup \pi_{C/S}(B)$
for any $A\in H^1(N_{S/\mathbb P}^{\vee}\big{\vert}_C \otimes N_{S/\mathbb P})$
and any $B\in H^0(N_{C/\mathbb P})$. Therefore,
\begin{align*}
  \ob_S(\alpha) &= (\pi_{C/S}(\alpha) \cup \mathbf e_C) \cup \alpha\\
  &= \pi_{C/S}^{\vee}((\beta \cup \mathbf e_S)\big{\vert}_C) \cup \alpha\\
  &= ((\beta \cup \mathbf e_S)\big{\vert}_C) \cup \pi_{C/S}(\alpha)\\
  &= ((\beta \cup \mathbf e_S)\big{\vert}_C) \cup \beta\big{\vert}_C\\
  &=\ob(\beta)\big{\vert}_C\qed.
\end{align*}

Let $E$ be an effective divisor on $S$ disjoint from $C$
(i.e.,~$C \cap E =\emptyset$).
Let $S^{\circ}:=S\setminus E$ and $\mathbb P^{\circ}:=\mathbb P\setminus E$
denote the complements of $E$ in $S$ and $\mathbb P$.
Every rational section $\beta$ of $N_{S/\mathbb P}$
having poles only along $E$ determines a global section 
$\beta^{\circ}$ of $N_{S^{\circ}/\mathbb P^{\circ}}$,
and the obstruction 
$\ob(\beta^{\circ}) \in H^1(N_{S^{\circ}/\mathbb P^{\circ}})$ to deforming
$S^{\circ}$ in $\mathbb P^{\circ}$ (See \S\ref{sec:with pole}).
Since $\ob(\alpha)$ is (and hence $\ob_S(\alpha)$ is) determined
by a neighborhood of $C$, we have the following
variant of Lemma~\ref{lem:reduction}.

\begin{lem}[{cf.~\cite[Lemma~2.12]{Nasu4}}]
  \label{lem:reduction2}
  Let $C \subset S \subset \mathbb P$ as in Lemma~\ref{lem:reduction},
  $E$ an effective divisor of $S$ disjoint from $C$,
  $\alpha$ a global section of $N_{C/\mathbb P}$.
  If there exists a rational section $\beta$ of $N_{S/\mathbb P}(E)$
  such that $\pi_{C/S}(\alpha)=\beta\big{\vert}_C$,
  then
  $$
  \ob_S(\alpha)=\ob(\beta^{\circ}) \big{\vert}_C
  $$
  in $H^1(N_{S/\mathbb P}\big{\vert}_C)$,
  where $\ob(\beta^{\circ})\big{\vert}_C$ is the image of
  $\ob(\beta^{\circ})$ by the map 
  $H^1(N_{S^{\circ}/\mathbb P^{\circ}}) 
  \overset{|_C}\longrightarrow 
  H^1(N_{S/\mathbb P}\big{\vert}_C)$,
  where $N_{S^{\circ}/\mathbb P^{\circ}}\big{\vert}_C \simeq N_{S/\mathbb P}\big{\vert}_C$.
\end{lem}

\subsection{Hilbert-flag schemes and relative Picard schemes}
\label{subsec:flag and picard}

In this subsection, we recall basic results on Hilbert-flag schemes and 
relative Picard schemes. For further details of the results in this subsection, 
we refer to \cite{Kleppe87,Kleppe81,Kleppe-Ottem15,Sernesi}, etc.

Let $X$ be a projective scheme, 
a pair $(C,S)$ of two closed subschemes of $X$ with 
Hilbert polynomials $(P,Q)$ and such that $C \subset S \subset X$.
Then there exists a projective scheme $\HF_{P,Q} X$,
called the {\em Hilbert-flag scheme} of $X$,
which parametrizes all pairs $(C',S')$ of closed subschemes of $X$ 
with Hilbert polynomials $(P,Q)$ such that $C' \subset S' \subset X$
(cf.~\cite{Kleppe87,Sernesi}).
Let $H_Q:=\Hilb_Q X$ be the Hilbert scheme of $X$ with Hilbert polynomial $Q$,
$\mathcal S$ the universal family 
$X\times H_Q \supset \mathcal S \rightarrow H_Q$.
Then we may view $\HF_{P,Q} X$ as a relative Hilbert scheme $\Hilb_P \mathcal S$
of $\mathcal S/H_Q$ (cf.~\cite[Theorem~24.7]{Hartshorne10}).
Let $\HF X$ denote the disjoint union
$\bigsqcup_{P,Q} \HF_{P,Q} X$,~i.e.,~the full Hilbert-flag scheme of $X$.
There exist two natural morphisms $pr_i: \HF X \rightarrow \Hilb X$ ($i=1,2$)
corresponding to projections 
$(C,S) \mapsto [C]$ and $(C,S) \mapsto [S]$, respectively.
Let $N_{(C,S)/X}$ denote the normal sheaf of the pair $(C,S)$ in $X$
that is defined by the fiber product of 
$N_{C/X} \overset{\pi_{C/S}}\longrightarrow N_{S/X}\big{\vert}_C$ and
$N_{S/X} \overset{|_C}{\longrightarrow} N_{S/X}\big{\vert}_C$
over $N_{S/X}\big{\vert}_C$ (cf.~\cite[Definition~4.5.2]{Sernesi}).
If the two embeddings
$C \hookrightarrow S$ and $S\hookrightarrow X$ are regular, then
\begin{enumerate}
  \item the cohomology groups $H^i(N_{(C,S)/X})$ ($i=0,1$) 
  respectively represent
  the tangent space and the obstruction space of $\HF X$ at $(C,S)$
  (cf.~\cite[Proposition~4.5.3]{Sernesi}).
  In particular, the vanishing of $H^1(N_{(C,S)/X})$ implies that
  $\HF X$ is nonsingular at $(C,S)$.
  \item there exist two natural exact sequences
  \begin{align}
    \label{ses:normal1}
    0 \longrightarrow \mathcal I_{C/S}\otimes_S N_{S/X} 
    \longrightarrow & N_{(C,S)/X} \overset{\pi_1}{\longrightarrow}
    N_{C/X} \longrightarrow 0\\
    \label{ses:normal2}
    0 \longrightarrow N_{C/S} \longrightarrow & N_{(C,S)/X} 
    \overset{\pi_2}{\longrightarrow} N_{S/X} \longrightarrow 0
  \end{align}
  on $X$, where each $\pi_i$ ($i=1,2$)
  induces the tangent map and the map on obstruction spaces of $pr_i$
  (cf.~\cite[\S2.2]{Nasu6}).
  \item \label{item:HF4} $pr_1$ is smooth at $(C,S)$ if
  $H^1(\mathcal I_{C/S}\otimes_S N_{S/X})=0$, and 
  so is $pr_2$ if $H^1(N_{C/S})=0$.
  (cf.~\cite[Lemma~A10]{Kleppe87}, \cite[Theorem~1.3.4]{Kleppe81}).
\end{enumerate}
Given a point $(C,S)$ of $\HF X$,
if $\HF X$ is nonsingular at a point $(C,S)$
and $pr_1$ is smooth at the same point,
then by a property of smooth morphisms, $pr_1$ maps
the unique irreducible component $\mathcal W_{C,S}$
of $\HF X$ passing through $(C,S)$
to the unique irreducible component $W_{C,S}$ of 
$\Hilb X$ passing through $[C]$.
Moreover, $\Hilb X$ is generically smooth along $W_{C,S}$
(cf.~\cite[Theorem~2.4]{Nasu6}).
The image $W_{C,S}$ of $\mathcal W_{C,S}$ by $pr_1$ is called
the {\em $S$-maximal family} of $\Hilb X$ containing $C$,
where its definition was originated from \cite{Mukai-Nasu}.
When $pr_1$ is not smooth, we consider an exact sequence
\begin{equation}
  \label{ses:first tangent map}
  \begin{CD}
    H^0(X,N_{(C,S)/X})
    \overset{p_1}{\longrightarrow} H^0(C,N_{C/X})
    \longrightarrow H^1(S,\mathcal I_{C/S}\otimes_S N_{S/X})
    \longrightarrow H^1(X,N_{(C,S)/X})
  \end{CD}
\end{equation}
of cohomology groups, which is induced by \eqref{ses:normal1}.
The following lemma is necessary for our proof of Theorem~\ref{thm:main1}.
\begin{lem}
  \label{lem:maximality}
  If $C \hookrightarrow S\hookrightarrow X$ are regular embeddings,
  $\HF X$ is nonsingular at $(C,S)$ and
  the primary obstruction $\ob(\alpha)$ is nonzero in $H^1(N_{C/X})$
  for every global section $\alpha \in H^0(N_{C/X})\setminus \im p_1$,
  then 
  \begin{enumerate}
    \item the $S$-maximal family $W_{C,S}$ is an irreducible 
    component of $(\Hilb X)_{\red}$, and
    \item if $H^0(\mathcal I_{C/S}\otimes_S N_{S/X})=0$,
    then
    $$
    \dim_{[C]} \Hilb X = \dim_{(C,S)} \HF X.
    $$
  \end{enumerate}
\end{lem}
\Proof See \cite[Lemma~2.17]{Nasu9} for the proof.
In fact, the assumption that $H^1(N_{(C,S)/X})=0$ of the same lemma
can be replaced by the nonsingularity of $\HF X$ at $(C,S)$. \qed

\medskip

In what follows,
we assume that $S$ and $X$ are nonsingular schemes, and 
$C \subset S$ is an effective Cartier divisor on $S$.
We consider the relative Picard schemes $\Pic \mathcal S$
of $\mathcal S/H_Q$.
Let $\Hilb^{cd} \mathcal S$ denote the relative Hilbert scheme of 
effective Cartier divisors on $\mathcal S/H_Q$,~i.e.,
the open subscheme of the Hilbert-flag scheme of $\HF X$
parametrizing all pairs $(C',S')$ of a closed subscheme $S' \subset X$
and an effective Cartier divisor $C' \subset S'$.
Then there exists a natural morphism
\begin{equation}
  \label{map: hilb to pic}
  \pi: \Hilb^{cd} \mathcal S \longrightarrow \Pic \mathcal S, \qquad
  (C',S') \mapsto ([\mathcal O_{S'}(C')],S'),
\end{equation}
where $\mathcal O_{S'}(C')$ is the invertible sheaf associated to $C'$, 
and $[\mathcal L]$ is the isomorphism class of
an invertible sheaf $\mathcal L$ on $S$.
Obviously, we have $pr_2\circ \pi=pr_2'$, where 
$pr_2': \Pic \mathcal S \rightarrow H_Q$ is a natural projection.
In what follows, we collect some infinitesimal properties of 
$\Hilb^{cd} \mathcal S$ and $\Pic \mathcal S$.
\begin{enumerate}
  \item \label{item:smoothness of pi}
  $\pi$ is smooth at $(C,S)$ if $H^1(S,\mathcal O_S(C))=0$
  (cf.~\cite[Remark~4.5]{TDTEGA5}).
  \item Let $\mathcal L=\mathcal O_S(C)$ in $\Pic S$ and let
  $T_{\Pic \mathcal S, ([\mathcal L],S)}$ denote the tangent space of $\Pic \mathcal S$
  at $([\mathcal L],S)$. Then there exists a commutative diagram
  \begin{equation}
    \label{diag:cohomology rel. hilb and pic}
    \begin{CD}
      0 @>>> H^0(N_{C/S}) @>>> H^0(N_{(C,S)/X}) @>{p_2}>> H^0(N_{S/X}) @ >{\delta_C}>> H^1(N_{C/S}) \\
      @. @VVV @V{\kappa}VV \Vert @. @VVV \\
      0 @>>> H^1(\mathcal O_S) @>>> T_{\Pic \mathcal S, ([\mathcal L],S)}
      @>>> H^0(N_{S/X}) @>{\delta_{\mathcal L}}>> H^2(\mathcal O_S)
    \end{CD}
  \end{equation}
  of exact sequences (cf.~\cite[\S2.2]{Kleppe-Ottem15}).
  Here the first horizontal sequence is the long exact sequence
  induced from \eqref{ses:normal2}, while 
  the first and the last vertical maps are 
  the coboundary maps induced by the exact sequence
  $0 \rightarrow \mathcal O_S \rightarrow \mathcal L
  \rightarrow N_{C/S} \rightarrow 0$ on $S$.
  The second vertical map $\kappa$
  represents the tangent map of $\pi$ at $(C,S)$, and
  if $H^1(S,\mathcal L)=0$, then $\kappa$ is surjective by the four lemma.
  If moreover, $H^1(N_{S/X})=0$, then the obstruction space of $\Hilb^{cd} \mathcal S$ at $(C,S)$
  is isomorphic to $\coker \delta_C$, and that of $\Pic \mathcal S$ at $([\mathcal L],S)$
  is isomorphic to $\coker \delta_{\mathcal L}$.
  Then by the four lemma again,
  $H^1(S,\mathcal L)=0$ implies that
  the map between the two obstruction spaces induced by $\pi$ is injective.
  Thus the first infinitesimal property \eqref{item:smoothness of pi} follows.
  \item By the first coboundary map of the exact sequence
  $0 \rightarrow T_S \rightarrow T_X\big{\vert}_S 
  \rightarrow N_{S/X} \rightarrow 0$,
  every global section $\beta$ of $N_{S/X}$ is mapped to
  an element $\tau$ in $H^1(T_S)$
  corresponding to a first-order deformation $\tilde S$ of 
  an abstract (non-embedded) variety $S$.
  Let $c(\mathcal L)\in H^1(\Omega^1_S)$
  ($\simeq \Ext^1(T_S,\mathcal O_S)$)
  denote the Atiyah extension class of $\mathcal L$,
  that is the image of $[\mathcal L] \in H^1(\mathcal O_S^\times)$
  ($\simeq \Pic S$) by the group homomorphism
  $c: H^1(\mathcal O_S^{\times})\rightarrow H^1(\Omega^1_S)$
  induced by the logarithmic derivative map 
  $d \log: \mathcal O_S^{\times} \rightarrow \Omega^1_S$, $u \mapsto du/u$
  (cf.~\cite[Chap.~V,~Ex.~1.8]{Hartshorne}, \cite[\S3.3.3]{Sernesi}).
  Then the coboundary map $\delta_{\mathcal L}$
  in the above diagram \eqref{diag:cohomology rel. hilb and pic}
  is defined by the cup product
  \begin{equation}
    \label{eqn:definition of delta}
    \delta_{\mathcal L}(\beta)=\tau \cup c(\mathcal L)
  \end{equation}
  via the natural cup product map
  $$H^1(S,T_S) \times H^1(S,\Omega^1_S) 
  \overset{\cup}\longrightarrow H^2(S,\mathcal O_S)$$ and
  $\mathcal L$ lifts to $\tilde S$ if and only if
  $\delta_{\mathcal L}(\beta)=0$ (cf.~\cite[Theorem~3.3.11]{Sernesi}).
  In particular, by the additive property
  $c(\mathcal L\otimes \mathcal M)=c(\mathcal L)+c(\mathcal M)$
  of the Atiyah extension classes, we have
  $\delta_{\mathcal L\otimes \mathcal M}=\delta_{\mathcal L}+\delta_{\mathcal M}$
  and $\delta_{\mathcal L^{\otimes n}}=n\delta_{\mathcal L}$
  for all $\mathcal L,\mathcal M \in \Pic S$ and any integer $n$.
  \item Let $\iota: S\hookrightarrow X$ denote 
  the closed embedding of $S$ into $X$. Then we have
  $\delta_{\mathcal L}=0$ if $\mathcal L$
  is contained in the image of $\iota^*: \Pic X \rightarrow \Pic S$
  (cf.~\cite[Lemma~2.11]{Nasu6}).
\end{enumerate}

The following lemma allows us to deduce the nonsingularity and 
local dimension of the Hilbert-flag scheme at $(C,S)$ from those at $(E,S)$, 
under a suitable linear equivalence between $C$ and $E$.

\begin{lem}
\label{lem:smoothness of flag1}
Let $X$ be a smooth projective scheme, 
$S \subset X$ a smooth closed subscheme such that $H^1(N_{S/X})=0$,
$\mathcal P$ an invertible sheaf on $X$, and
$C,E$ two effective Cartier divisors on $S$.
Suppose that there exist
two integers $a$ and $b$ such that $b\not\equiv 0 \mod p$, where
$p=\car k$, and 
$$
C \sim a \mathcal P\big{\vert}_S + b E \qquad \mbox{in $\Pic S$}.
$$
and we have $H^1(\mathcal O_S(C))=H^1(\mathcal O_S(E))=0$. Then
$\HF X$ is nonsingular at $(C,S)$ if and only if
$\HF X$ is nonsingular at $(E,S)$, and 
$$
\dim_{(C,S)} \HF X - \dim |C| = \dim_{(E,S)} \HF X - \dim |E|,
$$
where $|D|$ denotes the complete linear system on $S$
spanned by a divisor $D$ on $S$.
\end{lem}
\Proof
The proof is similar to that of \cite[Lemma~2.2]{Kleppe-Ottem15}.
Suppose that $\HF X$ is nonsingular at $(E,S)$.
Then it follows from $H^1(\mathcal O_S(E))=0$ that
$\pi: \Hilb^{cd} \mathcal S \rightarrow \Pic \mathcal S$ is smooth at $(E,S)$.
This implies that
$\Pic \mathcal S$ is nonsingular at $([\mathcal O_S(E)],S)$.
Then $\Pic \mathcal S$ is nonsingular at $([\mathcal O_S(C)],S)$, because
there exists an isomorphism
$\mathcal O_{\Pic \mathcal S, ([\mathcal O_S(E)],S)}
\simeq 
\mathcal O_{\Pic \mathcal S, ([\mathcal O_S(C)],S)}$ of local rings, 
up to completion. 
Indeed, by additivity and $c(\mathcal P\big{\vert}_S)=0$, we see that
$$
c(\mathcal O_S(C))=ac(\mathcal P\big{\vert}_S)+b c(\mathcal O_S(E))
=b c(\mathcal O_S(E)).
$$
Therefore, the two coboundary maps $\delta_{\mathcal O_S(C)}$ 
and $\delta_{\mathcal O_S(E)}$ coincide
up to the nonzero constant $0 \ne b \in k$, since $p=\car k \nmid b$.
Then it follows from a diagram similar to 
\eqref{diag:cohomology rel. hilb and pic} that
the morphism between the local deformation functors of 
$\Pic \mathcal S$ at $([\mathcal O_S(E)],S)$ 
and $\Pic \mathcal S$ at $([\mathcal O_S(C)],S)$
induced by the map
$\mathcal O_{S'}(E') \mapsto \mathcal P\big{\vert}_{S'}^{\otimes a} \otimes
\mathcal O_{S'}(E')^{\otimes b}$ is an isomorphism
(see also \cite[Proposition~2, Construction~2]{Ellingsrud-Peskine} 
and the proof of \cite[Lemma~2.2]{Kleppe-Ottem15}).
Since $\Pic \mathcal S$ is nonsingular at $([\mathcal O_S(C)],S)$,
$\HF X$ is nonsingular at $(C,S)$ by $H^1(\mathcal O_S(C))=0$.
Conversely, if $\HF X$ is nonsingular at $(C,S)$,
then we obtain the nonsingularity of $\HF X$
at $(E,S)$ by the same argument.
Since the fiber of $\pi$ in \eqref{map: hilb to pic} at
$([\mathcal O_S(C)],S)$ is clearly isomorphic to $|C|$ on $S$,
and since smooth morphisms have surjective tangent maps,
we also get the dimension formula in the lemma.
\qed

\section{Infinitesimal deformations with pole}
\label{sec:with pole}

This section is devoted to the proof of Theorem~\ref{thm:main2}.
We first recall basic notations and a result in \cite{Nasu4}.
Let $X$ be a smooth projective variety of dimension at least $3$,
$S$ a smooth surface in $X$,
$E$ a smooth connected curve in $S$.
We put $X^{\circ}:=X\setminus E$ and $S^{\circ}:=S\setminus E$, 
the complementary open subvarieties.
Let $\iota: S^{\circ} \hookrightarrow S$ be the open immersion.
Since $E$ is an effective divisor on $S$, for any integer $m \ge 0$,
there exist natural inclusions
$\mathcal O_S(mE) \hookrightarrow \iota_*\mathcal O_{S^{\circ}}$
and 
$N_{S/X}(mE) \hookrightarrow \iota_* N_{S^{\circ}/X^{\circ}}$ of sheaves.
The last inclusion induces
$$
H^0(S,N_{S/X}(mE)) \hookrightarrow 
H^0(S^{\circ},N_{S^{\circ}/X^{\circ}})
$$
and in particular, every rational section $\beta$ of $N_{S/X}$
having a pole only along $E$ and of order one,
that is, $\beta \in H^0(N_{S/X}(E))\setminus H^0(N_{S/X})$,
defines a global section $\beta^{\circ}$ of $N_{S^{\circ}/X^{\circ}}$.
Then $\beta^{\circ}$ corresponds to a first-order infinitesimal deformation 
of $S^{\circ}$ in $X^{\circ}$.
Such rational sections $\beta$ of $N_{S/X}$ 
have been studied in \cite{Mukai-Nasu,Nasu4,Nasu5}
and $\beta$ is called an {\em infinitesimal deformation with a pole}.
In what follows, we identify $\beta$ and $\beta^{\circ}$
by the above natural inclusions.

\begin{thm}[{\cite{Nasu4}}]\label{thm:open surface in 3-$fold}
  If $\dim X=3$,
  $E^2<0$, $\det N_{E/X}:=\bigwedge^2 N_{E/X}$ is trivial, and
  if the exact sequence
  \begin{equation}\label{ses:normal bundle of line}
    0 \longrightarrow N_{E/S} \longrightarrow 
    N_{E/X} \longrightarrow N_{S/X}\big{\vert}_E
    \longrightarrow 0
  \end{equation}
  does not split, then the first-order infinitesimal deformation of 
  $S^{\circ} \subset X^{\circ}$ induced by any element of
  $H^0(N_{S/X}(E))\setminus H^0(N_{S/X})$
  does not lift to a deformation over $\Spec k[t]/(t^3)$.
\end{thm}

Let $S$ be a complete intersection $S_{2,2,2} \subset \mathbb P^5$
containing a line $E \subset \mathbb P^5$.
We denote $\mathbb P^5$ by $\mathbb P$ in this section, 
and put $S^{\circ}:=S\setminus E$
and $\mathbb P^{\circ}:=\mathbb P\setminus E$.

\begin{lem}
  The inclusion $\iota: S^{\circ} \hookrightarrow S$ induces two injective maps
  $$
  H^1(S,\mathcal O_S(mE)) \hookrightarrow 
  H^1(S^{\circ},\mathcal O_{S^{\circ}}) \quad
  \mbox{and} \quad
  H^1(S,N_{S/\mathbb P}(mE)) \hookrightarrow 
  H^1(S^{\circ},N_{S^{\circ}/\mathbb P^{\circ}})
  $$
  of cohomology groups for any integer $m>0$.
\end{lem}
\Proof
We recall that
$\iota_* \mathcal O_{S^{\circ}}$ is the inductive limit
$\lim_{m \to \infty} \mathcal O_S(mE)$. 
Given an invertible sheaf $L$ on a projective surface $S$, there exists a natural inclusion
$L \hookrightarrow L \otimes \iota_* \mathcal O_{S^{\circ}}$ of sheaves on $S$.
It follows from \cite[Lemma~2.5]{Mukai-Nasu} that
if $E$ is a curve on $S$ satisfying $E^2<0$ and $\deg(L\big{\vert}_E)\le 0$,
then the map $H^1(S,L) \rightarrow H^1(S^{\circ}, L\big{\vert}_{S^{\circ}})$
induced by the above inclusion is injective.
Since $E$ is a rational curve on a $K3$ surface $S$, we see that $E^2=-2$,
and moreover, since
$N_{S/\mathbb P}\simeq \mathcal O_S(2)^{\oplus 3}$ by $S=S_{2,2,2}$,
we have the two natural filtrations
$$
H^1(S,\mathcal O_S(E)) \hookrightarrow 
H^1(S,\mathcal O_S(2E)) \hookrightarrow \dots 
\hookrightarrow H^1(S^{\circ},\mathcal O_{S^{\circ}})
$$
and
$$
H^1(S,N_{S/\mathbb P}(E)) \hookrightarrow 
H^1(S,N_{S/\mathbb P}(2E)) \hookrightarrow \dots 
\hookrightarrow H^1(S^{\circ},N_{S^{\circ}/\mathbb P^{\circ}})
$$
of cohomology groups on $S$ (or $S^{\circ}$). \qed

\medskip

Let $\mathbf e$ be the extension class of 
the exact sequence \eqref{ses:conormal}
for $S^{\circ} \subset \mathbb P^{\circ}$.
By Lemma~\ref{lem:primary obstruction},
in order to prove $\ob(\beta^{\circ})\ne 0$
for a given $\beta^{\circ} \in H^0(N_{S^{\circ}/\mathbb P^{\circ}})$,
it suffices to prove that the composition 
$$
\begin{CD}
  H^0(N_{S^{\circ}/\mathbb P^{\circ}}^{\vee}(2))
  @>{\cup \beta^{\circ}}>>
  H^0(\mathcal O_{S^{\circ}}(2))
  @>{\cup \mathbf e}>>
  H^1(N_{S^{\circ}/\mathbb P^{\circ}}^{\vee}(2))
  @>{\cup \mathbf \beta^{\circ}}>>
  H^1(\mathcal O_{S^{\circ}}(2)).
\end{CD}
$$
of cup product maps, which we denote by $\Phi(\beta^{\circ})$, is nonzero.
Let $q_i$ ($i=1,2,3$)
be the three quadratic polynomials defining $S$ in $\mathbb P$
defined in \S\ref{subsec:unprojection} (more precisely, see \eqref{eqn:S}).
Then for each $i$, the image $\bar q_i$ of $q_i$ by the natural 
surjection $\mathcal I_S \twoheadrightarrow \mathcal I_S/\mathcal I_S^2 
\simeq N_{S/\mathbb P}^{\vee}$
defines an element of $H^0(N_{S/\mathbb P}^{\vee}(2))$.
Then $\left\{\bar q_1,\bar q_2,\bar q_3\right\}$ 
is a free basis of
$N_{S/\mathbb P}^{\vee}(2) \simeq \mathcal O_S^{\oplus 3}$.
By $N_{S/\mathbb P}(E)\simeq \sHom(N_{S/\mathbb P}^{\vee},\mathcal O_S(E))$,
every global section $\beta$ of $N_{S/\mathbb P}(E)$ corresponds to 
an $\mathcal O_S$-linear homomorphism from $N_{S/\mathbb P}^{\vee}(2)$
to $\mathcal O_S(2)(E)$, which we denote by the same symbol $\beta$.
Since $N_{S/\mathbb P}^{\vee}$ is generated by $\bar q_i$,
$\beta$ is determined by
the three global sections $v_i:=\beta(\bar q_i)$ ($i=1,2,3$) of $\mathcal O_S(2)(E)$.
Therefore, we can identify $\beta \in H^0(N_{S/\mathbb P}(E))$
($\subset H^0(N_{S^{\circ}/\mathbb P^{\circ}})$)
with an element $(v_1,v_2,v_3)$ in 
$H^0(\mathcal O_S(2)(E))^{\oplus 3}$
($\subset H^0(\mathcal O_{S^{\circ}}(2))^{\oplus 3}$).
By abuse of notation, we write 
$$
\beta=(v_1,v_2,v_3).
$$

\begin{lem}
  \label{lem:equivalence}
  Let $\beta$ be a global section of $N_{S/\mathbb P}(E)$.
  Then the following are equivalent:
  \begin{enumerate}
    \item $\beta \not\in H^0(S,N_{S/\mathbb P})$,
    \item $\beta\big{\vert}_E \ne 0$ in 
    $H^0(E,N_{S/\mathbb P}(E)\big{\vert}_E)$,
    \item $v_i \not \in H^0(S,\mathcal O_S(2))$ for some $i$, and
    \item $v_i\big{\vert}_E \ne 0$ in $H^0(E,\mathcal O_S(2)(E)\big{\vert}_E)$
    for some $i$.
  \end{enumerate}
\end{lem}
\Proof The lemma immediately follows from the exact sequence
$$
0 \longrightarrow \underbrace{N_{S/\mathbb P}}_{\simeq \mathcal O_S(2)^{\oplus 3}}
\longrightarrow \underbrace{N_{S/\mathbb P}(E)}_{\simeq \mathcal O_S(2)(E)^{\oplus 3}}
\overset{|_E}{\longrightarrow} \underbrace{N_{S/\mathbb P}(E)\big{\vert}_E}_{\simeq \mathcal O_E^{\oplus 3}}
\longrightarrow 0.\qed
$$

We consider an exact sequence
$$
0 \rightarrow N_{S^{\circ}/\mathbb P^{\circ}}^{\vee}(2) 
\rightarrow \mathcal O_{\mathbb P^{\circ}}/\mathcal I_{S^{\circ}}^2(2)
\rightarrow \mathcal O_{S^{\circ}}(2) \rightarrow 0,
$$
which is obtained by 
$\eqref{ses:conormal} \otimes \mathcal O_{\mathbb P^{\circ}}(2)$ for
$S^{\circ} \subset \mathbb P^{\circ}$.
Let $\delta: H^0(\mathcal O_{S^{\circ}}(2))
\rightarrow H^1(N_{S^{\circ}/\mathbb P^{\circ}}^{\vee}(2))$
denote the first coboundary map of the 
long exact sequence associated to this sequence.
Let $\theta$ be a global section of $\mathcal O_S(2)(E)$
defined in \eqref{eqn:section with pole},
which defines the unprojection morphism of $E$ on $S$
(see \S~\ref{subsec:unprojection}).
The following lemma is a key to our proof of Theorem~\ref{thm:main2}.
\begin{lem}
  \label{lem:coboundary image}
  The coboundary $\delta(\theta)$ of $\theta$ is contained in a subgroup
  $$
  H^1(S,N_{S/\mathbb P}^{\vee}(2)(2E))
  \hookrightarrow
  H^1(S,N_{S^{\circ}/\mathbb P^{\circ}}^{\vee}(2)).
  $$
  In fact, under the basis $\left\{\bar q_1,\bar q_2,\bar q_3\right\}$
  of $N_{S/\mathbb P}^{\vee}(2)\simeq \mathcal O_S^{\oplus 3}$, we have
  \begin{equation}
    \label{eqn:coboundary of delta}
    \delta(\theta)=-\sum_{k=1}^3 t_k \bar q_k
  \end{equation}
  in $H^1(S,N_{S/\mathbb P}^{\vee}(2)(2E))$, 
  where $t_k$ ($k=1,2,3$) are the cohomology classes in 
  $H^1(\mathcal O_S(2E))$
  defined in \eqref{eqn:fundamental 1-cocycles}.
\end{lem}
\Proof
In \S\ref{subsec:unprojection}, we have seen
that the open affine subsets $D(x_i)$ ($0 \le i\le 3$) of $\mathbb P^{\circ}$
cover $\mathbb P^{\circ}$ (and hence $S^{\circ}$).
We compute the coboundary $\delta(\theta)$ of $\theta$
by using the {\v C}ech cohomology with respect to
the covering $\mathfrak U=\left\{D(x_i)\bigm| 0 \le i\le 3\right\}$ 
of $\mathbb P^{\circ}$. Let $\hat \theta_i=\Delta_i/x_i$ in
$\Gamma(D(x_i),\mathcal O_{\mathbb P^{\circ}}/\mathcal I_{S^{\circ}}^2(2))$
be the local lift of $\theta$ over $D(x_i)$.
Then by Lemma~\ref{lem:equality of ratio}, we have 
\begin{align*}
x_ix_j(\hat \theta_j - \hat \theta_i) &=
x_i \Delta_j - x_j \Delta_i\\
&= (-1)^{i+j} \det (\mathbf q | A_{ij}) \\
&= -\sum_{k=1}^3 \Delta_{k;i,j} q_k,
\end{align*}
where $\Delta_{k;i,j}$ are the $2$-minors of $A$ defined in 
\eqref{eqn:(n-1)-minors}. 
By dividing with $x_ix_j$ and restricting it to $S^{\circ}$,
we see that for every $i,j$, 
$$
\hat \theta_j -\hat \theta_i 
=-\sum_{k=1}^3 (\Delta_{k;i,j}/x_ix_j) \bar q_k
$$
in $\Gamma(D(x_i)\cap D(x_j),N_{S^{\circ}/\mathbb P^{\circ}}^{\vee}(2))$.
Therefore, the $1$-cocycle $(\hat \theta_j -\hat \theta_i)_{ij}$
representing $\delta(\theta)$ 
in $H^1(N_{S^{\circ}/\mathbb P^{\circ}}^{\vee}(2))$
is equal to the right hand side of \eqref{eqn:coboundary of delta}.
Since $t_k$ and $\bar q_k$ are respectively contained in 
$H^1(\mathcal O(2E))$ and $H^0(N_{S/\mathbb P}^{\vee}(2))$ for any $k$,
$\delta(\theta)$ is contained in 
the subgroup $H^1(N_{S/\mathbb P}^{\vee}(2)(2E))$. \qed

\medskip

\paragraph{\bf Proof of Theorem~\ref{thm:main2}}

Let $\beta$ be a global section of 
$N_{S/\mathbb P}(E)\simeq \sHom(N_{S/\mathbb P}^{\vee},\mathcal O_S(E))$
not contained in $H^0(N_{S/\mathbb P})\simeq \Hom(N_{S/\mathbb P}^{\vee},\mathcal O_S)$.
Then by Lemma~\ref{lem:equivalence},
there exists a quadratic form $q$ in $H^0(\mathcal I_S(2))$
such that $v:=\beta(\bar q)$
is not contained in $H^0(\mathcal O_S(2))$ but in $H^0(\mathcal O_S(2)(E))$.
Since $\mathcal O_S(2)(E)\big{\vert}_E\simeq \mathcal O_E$,
there exists a nonzero constant $c \in k^{\times}$
such that $v\big{\vert}_E=c \theta\big{\vert}_E$,
in other words, we have $v=c\theta$ up to an element of $H^0(\mathcal O_S(2))$.
Since the coboundary map $\delta$ is $k$-linear, we have
$$
v\cup \mathbf e=\delta (v) =c \delta(\theta)
$$
in $H^1(S,N_{S/\mathbb P}^{\vee}(2)(2E))$.
Therefore, it follows from Lemma~\ref{lem:coboundary image} that
$$
\Phi(\beta)(\bar q)=\ob(\beta)(\bar q)=(\beta \cup \mathbf e \cup \beta)(\bar q)
=c \delta(\theta) \cup \beta=-c\sum_{k=1}^3 t_k \cup v_k
$$
in $H^1(S,\mathcal O_S(2)(3E))$, where
the cup product of $t_k \cup v_k$ is taken by the cup product map
$$
H^1(S,\mathcal O_S(2E)) \times 
H^0(S,\mathcal O_S(2)(E))
\overset{\cup}{\longrightarrow}
H^1(S,\mathcal O_S(2)(3E)).
$$
The last cup product map is compatible with 
the cup product map
$$
H^1(E,\mathcal O_E(2E)) \times 
H^0(E,\mathcal O_E(2)(E))
\overset{\cup}{\longrightarrow}
H^1(E,\mathcal O_E(2)(3E)).
$$
on $E$ via the restrictions to $E$.
We note that $\mathcal O_E(2)(E)^{\oplus 3}$ is trivial sheaf on 
$E \simeq \mathbb P^1$ and 
$(v_1\big{\vert}_E,v_2\big{\vert}_E,v_3\big{\vert}_E)\ne (0,0,0)$
by Lemma~\ref{lem:equivalence} again.
Thus
$$
\ob(\beta)(\bar q)\big{\vert}_E=
-c\sum_{k=1}^3 t_k\big{\vert}_E \cup v_k\big{\vert}_E\ne 0 
$$
by Lemma~\ref{lem:linear independence of coh classes}.
This implies that $\ob(\beta)\ne 0$ and the proof has been completed.
\qed

\section{Obstructions to deforming curves on $S_{2,\dots,2}$}
\label{sec:obstruction}

The purpose of this section is to prove Theorem~\ref{thm:main1}.

Let $n \ge 3$ be an integer, let $\mathbb P:=\mathbb P^{n+2}$
and let $E$ be a line in $\mathbb P$.
As is well known, lines in $\mathbb P$ are parametrised by 
the Grassmannian $\Gr(2,n+3)$
($\simeq \Gr(n+1,H^0(\mathbb P,\mathcal O_{\mathbb P}(1)))$) 
of dimension $2n+2$.
Let $S:=S_{2,\dots,2}$ be a smooth complete intersection of 
$n$ quadrics in $\mathbb P$.
By the Riemann--Roch theorem for surfaces, the Hilbert polynomial of 
$S$ is computed as 
$$
\begin{aligned}
P_S(t)
&=\chi(\mathcal O_S(t\mathbf h))\\
&=\frac{1}{2}(t\mathbf h)\cdot(t\mathbf h-K_S)
  +\chi(\mathcal O_S)\\
&=2^{n-1}t(t-n+3)
  +2^{n-3}(n^2-5n+8),
\end{aligned}
$$
where we have used $\mathbf h^2=2^n$, $K_S=(n-3)\mathbf h$, and
\begin{equation}
  \label{eqn:euler char}
  \chi(\mathcal O_S)=2^{n-3}(n^2-5n+8).
\end{equation}
Indeed, $c(T_S)=(1+\mathbf h)^{n+3}/(1+2\mathbf h)^n$
and Noether's formula give the above formula for the
Euler characteristic $\chi(\mathcal O_S)$ of $S$. 
Let $P(t):=2^{n-1}t(t-n+3)+2^{n-3}(n^2-5n+8)$.

\begin{lem}
  \label{lem:smoothness of flag2}
  The Hilbert-flag scheme 
  $\HF^{ci}_{t+1, P(t)} \mathbb P$
  parametrizing flags 
  $$
  E \subset S \subset \mathbb P
  $$ of a line $E$ and a smooth complete intersection
  $S=S_{2,\dots,2}$ in $\mathbb P$
  is isomorphic to an open subset $\mathcal U$ of 
  $\Gr(n,(n+1)(n+6)/2)$-bundle over $\Gr(2,n+3)$.
  In particular, $\HF^{ci}_{t+1, P(t)} \mathbb P$ is irreducible and
  smooth along $\mathcal U$, and
  $$
  \dim \HF^{ci}_{t+1, P(t)} \mathbb P
  =\dfrac{n(n+2)(n+3)}2+2n+2.
  $$
\end{lem}
\Proof Given a line $E \subset \mathbb P$,
the linear system $|\mathcal I_E(2)| \subset |\mathcal O_{\mathbb P}(2)|$
of quadrics containing $E$ is isomorphic to 
the projective space $\mathbb P^N$, 
where $N+1=h^0(\mathcal O_{\mathbb P}(2))-h^0(\mathcal O_E(2))=(n+1)(n+6)/2$.
Since the fiber of the projection
$$
pr_1: \HF^{ci}_{t+1,P(t)} \mathbb P \rightarrow \Gr(2,n+3), \quad (E,S) \mapsto [E]
$$
at $[E]$ is isomorphic to an open subset of $\Gr(n,H^0(\mathcal I_E(2)))$,
we obtain the first conclusion. 
Finally we obtain the dimension of $\HF^{ci}_{t+1,P(t)} \mathbb P$
by its fiber-bundle structure.
\qed

\begin{rmk}
\label{rmk:smoothness of flag4}
In Lemma~\ref{lem:smoothness of flag2}, we use a global argument to show that
the Hilbert-flag scheme $\HF^{ci} \mathbb P$ is nonsingular at the point $(E,S)$
for every flag $E \subset S \subset \mathbb P$ of
a line $E$ and a smooth complete intersection $S=S_{2,\dots,2}$.
One can also use a direct computation to prove this fact.
Indeed, $H^2(\mathcal I_S(2))\simeq H^1(\mathcal O_S(2))=0$
by Lemma~\ref{lem:vanishing of H^1_*O_S}.
Since $E \subset \mathbb P$ is projectively normal,
it follows from the exact sequence
$0 \rightarrow \mathcal I_S(2) \rightarrow \mathcal I_E(2)
\rightarrow \mathcal O_S(2\mathbf h-E) \rightarrow 0$
that $H^1(N_{S/\mathbb P}(-E))
\simeq H^1(\mathcal O_S(2\mathbf h-E)^{\oplus n})=0$.
Since $N_{E/\mathbb P} \simeq \mathcal O_E(1)^{\oplus n+1}$, 
we have $H^1(N_{E/\mathbb P})=0$.
Thus it follows from the exact sequence
$
0 \rightarrow N_{S/\mathbb P}(-E) \rightarrow 
N_{(E,S)/\mathbb P} \overset{\pi_1}{\rightarrow} N_{E/\mathbb P} \rightarrow 0
$
that 
$H^1(N_{(E,S)/\mathbb P})=0$, and hence $\HF \mathbb P$ is nonsingular at
$(E,S)$.
The dimension formula for 
$\HF^{ci}_{t+1,P(t)} \mathbb P$ (or equivalently, $H^0(N_{(E,S)/\mathbb P})$)
also follows from the exact sequence
$
0 \rightarrow N_{E/S}
\rightarrow N_{(E,S)/\mathbb P} 
\overset{\pi_2}{\rightarrow} N_{S/\mathbb P} 
\rightarrow 0,
$
together with $N_{S/\mathbb P}\simeq \mathcal O_S(2)^{\oplus n}$ and 
$N_{E/S}\simeq \mathcal O_E(1-n)$ (cf.~\S\ref{subsec:flag and picard}).
Since $H^1(N_{E/S})\ne 0$ by $n\ge 3$,
there exists a first-order deformation $\tilde S$ of 
$S$ in $\mathbb P$, to which the closed subscheme $E$ of $S$ does not lift.
\end{rmk}

Let $C$ be a smooth connected curve on $S=S_{2,\dots,2}$ such that
\begin{equation}
  \label{eqn:class of C}
  C \sim mD, \quad D:=(n-1)\mathbf h +E, \quad m \in \mathbb Z_{>0}.
\end{equation}
The following lemma shows that the local deformation space of the 
pair $(C,S)$ in $\mathbb P^{n+2}$ is well behaved.

\begin{lem}
  \label{lem:smoothness of flag3}
  If $m$ is sufficiently large and coprime to
  the characteristic $p = \car k$ of the ground field,
  then $\HF \mathbb P$ is nonsingular at $(C,S)$
  of dimension
  $$
  \dfrac{n(n+2)(n+3)}2+2n+2+\dim |mD|.
  $$
\end{lem}
\Proof
By Serre duality, we see that 
$H^1(\mathcal O_S(E))\simeq H^1(\mathcal O_S(n-3)(-E))^{\vee}$,
where the last cohomology group vanishes, because
$H^1(\mathcal I_{E/S}(n-3))\simeq H^1(\mathcal I_E(n-3))=0$
by the projective normality of $E \subset \mathbb P$.
Since $m\gg 0$, 
we have $H^1(\mathcal O_S(C))=0$ by Lemma~\ref{lem:vanishing-D}.
Therefore $\HF \mathbb P$ is nonsingular at $(C,S)$
by Lemmas \ref{lem:smoothness of flag1} and \ref{lem:smoothness of flag2}.
Finally we obtain the local dimension of $\HF \mathbb P$ at $(C,S)$
by the same lemma, because $\dim|E|=0$ and $|C|=|mD|$.
\qed

\begin{rmk}
  \label{rmk:formula for dim|mD|}
  In Lemma~\ref{lem:smoothness of flag3},
  $\dim|mD|$ depends only on two integers $m$ and $n$. 
  Indeed, since $D$ is nef and big and $m\gg 0$,
  $K_S-mD$ is not effective.
  Then by Serre duality, we see that
  $H^2(S,mD) \simeq H^0(S,K_S-mD)^{\vee}=0$.
  This implies that $h^0(S,mD)=\chi(S,mD)$ by Lemma~\ref{lem:vanishing-D}.
  Then by \eqref{eqn:euler char},
  it follows from the Riemann-Roch theorem on surfaces that
  $\dim |mD|$ is equal to the number
  \begin{equation}
    \label{eqn:formula for dim|mD|}
    \lambda_{n,m}:=
    \dfrac{m(2^n(n-1)+1))(m(n-1)-n+3)}2+2^{n-3}(n^2-5n+8)-1.
  \end{equation}
\end{rmk}

Let $\HF^{sc} \mathbb P$ denote the inverse image of 
$\Hilb^{sc} \mathbb P$ by the first projection 
$pr_1: \HF \mathbb P \to \Hilb \mathbb P$,
and let $pr_1'$ denote the restriction of $pr_1$ to $\HF^{sc} \mathbb P$.

\begin{prop}
  \label{prop:generic smoothness}
  Suppose that $n>3$, $p = \car k \nmid m$ and
  $m$ is sufficiently large.
  \begin{enumerate}
    \item The first projection
    $pr_1': \HF^{sc} \mathbb P \to \Hilb^{sc} \mathbb P$ is smooth at $(C,S)$,
    and hence $\Hilb^{sc} \mathbb P$ is nonsingular at $[C]$.
    \item The $S$-maximal family $W_{C,S}$ of $\Hilb^{sc} \mathbb P$
    containing $[C]$
    becomes a generically smooth component of $\Hilb^{sc} \mathbb P$
    of dimension 
    $$
    \dfrac{n(n+2)(n+3)}2+2n+2+\lambda_{n,m},
    $$
    where $\lambda_{n,m}$ is the number defined in \eqref{eqn:formula for dim|mD|}.
  \end{enumerate}
\end{prop}
\Proof
Lemma~\ref{lem:vanishing-D} shows that $H^1(S,2\mathbf h-C)=0$ by $n>3$.
Since $N_{S/\mathbb P}\simeq \mathcal O_S(2)^{\oplus n}$,
we see that $H^1(S,N_{S/\mathbb P}(-C))=0$. 
Then it follows from the exact sequence \eqref{ses:normal1}, 
or directly from \S\ref{subsec:flag and picard} \eqref{item:HF4}, 
that the morphism
$pr_1: \HF \mathbb P \to \Hilb \mathbb P$,
which corresponds to the first projection, is smooth at $(C,S)$.
By the smoothness, (1) follows.

By Lemma~\ref{lem:smoothness of flag3},
there exists a unique irreducible component $\mathcal W_{C,S}$ 
of $\HF^{sc} \mathbb P$ passing through $(C,S)$.
Since every smooth morphism maps a generic point to a generic point 
\cite[Proposition~9.7,~III]{Hartshorne},
$\mathcal W_{C,S}$ dominates 
an irreducible component $W_{C,S}$ of $\Hilb^{sc} \mathbb P$,
where $W_{C,S}$ is the $S$-maximal family
of $\Hilb^{sc} \mathbb P$ containing $[C]$.
We note that $\Hilb^{sc} \mathbb P$ is smooth along an open dense subset
of $W_{C,S}$ by (1). Since
$pr_1': \HF^{sc} \mathbb P \rightarrow  \Hilb^{sc} \mathbb P$ 
induces a map 
$p_1: H^0(N_{(C_{\eta},S_{\eta})/\mathbb P}) \to H^0(N_{C_{\eta}/\mathbb P})$ 
from the tangent space of $\HF^{sc} \mathbb P$
to that of $\Hilb^{sc} \mathbb P$
at the generic point $(C_{\eta},S_{\eta})$ of $\mathcal W_{C,S}$,
and this map is an isomorphism by $H^i(N_{S/\mathbb P}(-C))=0$ ($i=0,1$),
(2) follows from Lemma~\ref{lem:smoothness of flag3}.
\qed

\medskip 

In what follows, we assume
that $n=3$,~i.e.,~$S=S_{2,2,2} \subset \mathbb P=\mathbb P^{5}$.
We need the following lemma.
\begin{lem}
  \label{lem:exterior deformations}
  Let $S$ be a smooth complete intersection surface
  $S_{2,2,2} \subset \mathbb P^5$,
  $C$ a smooth connected curve on $S$ such that
  $$
  C \sim mD, \quad D:=2\mathbf h +E,\quad m \gg 0, \quad \mbox{and} 
  \quad \car k=p \nmid m.
  $$
  Then
  \begin{enumerate}
    \item \label{item:vanishing of H^0 and H^1}
    $H^i(S,N_{S/\mathbb P}(E-C))=0$ for $i=0,1$.
    \item \label{item:unobstructed of flag}
    $H^1(\mathbb P,N_{(C,S)/\mathbb P})=0$.
    \item \label{item:cokernel of tangent map}
    The tangent map
    $$
    p_1: T_{\HF \mathbb P^5, (C,S)} 
    \longrightarrow  
    T_{\Hilb \mathbb P^5, [C]}=H^0(C,N_{C/\mathbb P^5})
    $$
    of the first projection $pr_1$ at $(C,S)$ is injective
    and $\coker p_1$ is isomorphic to
    $$
    H^1(S,N_{S/\mathbb P}(-C)) \simeq k^3.
    $$
  \end{enumerate}
\end{lem}
\Proof
\noindent
We note that $2\mathbf h+E-C\sim (1-m)D$.
Then it follows from Lemma~\ref{lem:vanishing-D} that
$H^1(S,(1-m)D)=0$ by $m \gg 0$, while
$H^0(S,(1-m)D)=0$ because $D$ is nef and big.
Since 
$N_{S/\mathbb P}(E-C)\simeq \mathcal O_S((1-m)D)^{\oplus 3}$,
we obtain \eqref{item:vanishing of H^0 and H^1}.

We prove \eqref{item:unobstructed of flag}.
It follows from Lemma~\ref{lem:vanishing-D} that $H^1(\mathcal O_S(C))=0$.
Since $H^1(N_{S/\mathbb P})=0$,
the commutative diagram \eqref{diag:cohomology rel. hilb and pic}
shows that 
$$
H^1(\mathbb P,N_{(C,S)/\mathbb P})\simeq \coker \delta_C 
\hookrightarrow \coker \delta_{\mathcal O_S(C)}.
$$
Moreover, since $S$ is $K3$, the target space
$H^2(S,\mathcal O_S)$ of $\delta_{\mathcal O_S(C)}$
is of dimension one. Thus 
it suffices to show that 
the map $\delta_{\mathcal O_S(C)}$ is nonzero.
We recall the definition \eqref{eqn:definition of delta}
of $\delta_{\mathcal L}$ and its additive property,~i.e.,
$\delta_{\mathcal L \otimes \mathcal \mathcal M}
=\delta_{\mathcal L}+\delta_{\mathcal M}$ 
for any $\mathcal M,\mathcal L \in \Pic S$.
We also have $\delta_{\mathcal L}=0$ for any
$\mathcal L \in \im [\iota^*:\Pic \mathbb P^5 \to \Pic S]$.
Therefore, $\delta_{\mathcal O_S(C)}=m\delta_{\mathcal O_S(E)}$
by $C \sim mD$ and $D=2\mathbf h +E$.
Then $\delta_{\mathcal O_S(E)}\ne 0$.
Indeed, there exists a first-order deformation $\tilde S$ of $S$
in $\mathbb P$, to which the closed subscheme 
$E$ of $S$ does not lift
(cf.~Remark~\ref{rmk:smoothness of flag4}).
This implies that neither does $\mathcal O_S(E)$ by $H^1(\mathcal O_S(E))=0$.
Therefore, we obtain $\delta_{\mathcal O_S(C)} \ne 0$ by $p \nmid m$.

Finally we prove \eqref{item:cokernel of tangent map}.
Since $D\cdot E=0$, the sheaf $N_{S/\mathbb P}(E-C)\big{\vert}_E$ 
on $E$ is the trivial free sheaf of rank three.
Therefore, there exists an exact sequence
$$
0 \longrightarrow 
N_{S/\mathbb P}(-C)
\longrightarrow
N_{S/\mathbb P}(E-C)
\overset{|_E}{\longrightarrow}
\mathcal O_E^{\oplus 3}
\longrightarrow
0
$$
on $S$. We deduce $H^1(N_{S/\mathbb P}(-C))\simeq k^3$
and $H^0(N_{S/\mathbb P}(-C))=0$
from this sequence and \eqref{item:vanishing of H^0 and H^1}.
Then it follows from \eqref{ses:normal1} and 
\eqref{item:unobstructed of flag} that
there exists an exact sequence
$$
0 \longrightarrow H^0(\mathbb P,N_{(C,S)/\mathbb P})
\overset{p_1}\longrightarrow H^0(C,N_{C/\mathbb P})
\longrightarrow H^1(S,N_{S/\mathbb P}(-C)) \simeq k^3
\longrightarrow 0, 
$$
and thereby \eqref{item:cokernel of tangent map} follows. \qed

\medskip

The following proposition shows that 
every first-order infinitesimal deformation of $C$ in $\mathbb P^5$ 
that does not arise from a deformation of the flag $C\subset S\subset\mathbb P^5$ 
is obstructed at the second order.
\begin{prop}
  \label{prop:obstruction}
  Let $S$ be a smooth complete intersection 
  $S_{2,2,2} \subset \mathbb P^5$,
  $C$ a smooth connected curve on $S$ such that
  $$
  C \sim mD, \quad D:=2\mathbf h +E,\quad m \gg 0, \quad \mbox{and} 
  \quad \car k=p \nmid m.
  $$
  Let $\alpha$ be a global section of $N_{C/\mathbb P^5}$,
  and $p_1$ the tangent map of the first projection
  $pr_1: \HF \mathbb P^5 \rightarrow \Hilb \mathbb P^5$
  at $(C,S)$ in Lemma~\ref{lem:exterior deformations}.
  If $\alpha$ is not contained in $\im p_1$,
  then
  $\ob(\alpha)\ne 0$ in $H^1(N_{C/\mathbb P^5})$.
\end{prop}
\Proof
The proof is similar to those 
of \cite[Proposition~4.10]{Nasu4} and \cite[Theorem~4.4]{Nasu10}.
Let $\alpha$ be a global section of $N_{C/\mathbb P}$,
not contained in $\im p_1$.
We consider the exterior components 
$u:=\pi_{C/S}(\alpha) \in H^0(N_{S/\mathbb P}\big{\vert}_C)$ 
and $\ob_S(\alpha) \in H^1(N_{S/\mathbb P}\big{\vert}_C)$ 
of $\alpha$ and $\ob(\alpha)$, respectively
(see ~\S\ref{subsec:infinitesimal}).
We prove that $\ob_S(\alpha)\ne 0$ instead of $\ob(\alpha)\ne 0$.
By \eqref{ses:normal1}, there exists a commutative diagram
$$
\begin{CD}
  0 @>>> N_{S/\mathbb P}(-C) 
  @>>> N_{(C,S)/\mathbb P} @>{\pi_1}>>
  N_{C/\mathbb P} @>>> 0\\
  @. \Vert @. @V{\pi_2}VV @V{\pi_{C/S}}VV  @. \\
  0 @>>> N_{S/\mathbb P}(-C) 
  @>>> N_{S/\mathbb P} @>{|_C}>>
  N_{S/\mathbb P}\big{\vert}_C @>>> 0
\end{CD}
$$
of exact sequences of normal sheaves of $C$, $S$ and $(C,S)$.
Since $\alpha\not\in \im p_1$, $u=\pi_{C/S}(\alpha)$ is not contained in the image of 
$H^0(N_{S/\mathbb P}) \rightarrow H^0(N_{S/\mathbb P}\big{\vert}_C)$.
This implies that there exist no global sections $\beta'$ of $N_{S/\mathbb P}$
such that $\beta'\big{\vert}_C=u$.
Tensoring the lower horizontal sequence with $\mathcal O_S(E)$
and taking a long exact sequence of cohomology groups, 
we obtain an exact sequence
$$
H^0(S,N_{S/\mathbb P}(E)) \overset{|_C}{\longrightarrow}
H^0(C,N_{S/\mathbb P}(E)\big{\vert}_C) \longrightarrow
H^1(S,N_{S/\mathbb P}(E-C)),
$$
where the last cohomology group 
vanishes by Lemma~\ref{lem:exterior deformations}.
Therefore,
there exists a global section $\beta$ of $N_{S/\mathbb P}(E)$
such that 
$$
\beta\big{\vert}_C=\bar u \quad 
\mbox{in} \quad H^0(N_{S/\mathbb P}(E)\big{\vert}_C).
$$
Here 
there exists a natural inclusion
$H^0(N_{S/\mathbb P}\big{\vert}_C) \hookrightarrow 
H^0(N_{S/\mathbb P}(E)\big{\vert}_C)$ since $E$ is an effective divisor on $S$,
and $\bar u$ denotes the image of $u$ by this map.
Moreover, $\beta$ is not contained in $H^0(N_{S/\mathbb P})$,
otherwise, by the above commutative diagram,
$\alpha$ is contained in $\im p_1$.
Thus $\beta$ defines {\em an infinitesimal deformation with a pole} along $E$
(see \S\ref{sec:with pole}).

Let $S^{\circ}:=S\setminus E$ and 
$\mathbb P^{\circ}:=\mathbb P\setminus E$.
Then there exists a natural injection
$H^0(S,N_{S/\mathbb P}(E)) \hookrightarrow
H^0(S^{\circ},N_{S^{\circ}/\mathbb P^{\circ}})$ of cohomology groups.
In what follows, we identify $\beta \in H^0(N_{S/\mathbb P}(E))$ 
with its image $\beta^{\circ}$ in $H^0(N_{S^{\circ}/\mathbb P^{\circ}})$. 
Since $\beta \in H^0(N_{S/\mathbb P}(E))\setminus H^0(N_{S/\mathbb P})$,
we obtain $\ob(\beta)\ne 0$ in $H^1(N_{S^{\circ}/\mathbb P^{\circ}})$
by virtue of Theorem~\ref{thm:main2}.
In fact, the proof of the same theorem shows that
$\ob(\beta)$ is contained in 
a smaller subgroup $H^1(S,N_{S/\mathbb P}(3E))$ of
$H^1(S^{\circ},N_{S^{\circ}/\mathbb P^{\circ}})$
and $\ob(\beta)\big{\vert}_E$ is nonzero in 
$H^1(E,N_{S/\mathbb P}(3E)\big{\vert}_E)\simeq 
H^1(E,\mathcal O_E(-4)^{\oplus 3})$.
Then it follows from Lemma~\ref{lem:reduction2} that
$$
\ob_S(\alpha) =\ob(\beta)\big{\vert}_C
$$
in $H^1(N_{S/\mathbb P}\big{\vert}_C)$,
because $C$ is contained in $S^{\circ}=S\setminus E$.

Finally we show that 
$\ob(\beta)\big{\vert}_C\ne 0$ 
in $H^1(C,N_{S/\mathbb P}\big{\vert}_C)$.
There exists an exact sequence 
$$
0 \longrightarrow N_{S/\mathbb P}(3E-C)
\longrightarrow N_{S/\mathbb P}(3E) 
\overset{|_C}\longrightarrow N_{S/\mathbb P}\big{\vert}_C
\longrightarrow 0
$$
on $S$. 
We note by $D=2\mathbf h+E$ and $C\sim mD$ that 
$$
N_{S/\mathbb P}(3E-C) 
\simeq \mathcal O_S(2)(3E-C)^{\oplus 3}
\simeq \mathcal O_S(2E-(m-1)D)^{\oplus 3}.
$$
Thus $H^1(S,N_{S/\mathbb P}(3E-C))=0$ by Lemma~\ref{lem:vanishing-D} 
and Serre duality on the $K3$ surface $S$.
This implies the restriction map 
$H^1(S,N_{S/\mathbb P}(3E)) \rightarrow H^1(C,N_{S/\mathbb P}\big{\vert}_C)$
is injective.
Therefore we have $\ob_S(\alpha)=\ob(\beta)\big{\vert}_C\ne 0$.
\qed

\medskip

Now we prove our main theorem.

\paragraph{\bf Proof of Theorem~\ref{thm:main1}}
Let $S,\mathbb P,E,\Lambda$ and $D$ be as in the theorem, 
and let $C$ be a general member of the complete linear system $\Lambda=|mD|$. 
By Lemma~\ref{lem:unprojection}, the linear system $|D|$ defines a birational morphism 
$f:S\to\overline S\subset\mathbb P^N$ contracting $E$ to a point $q$. 
Since $D\cdot E=0$, a general member of $|D|$ is disjoint from $E$ 
and is isomorphic to a hyperplane section of $\overline S$ not passing through $q$. 
Similarly, $C$ is disjoint from $E$ and corresponds to 
a general member of the linear system $|\mathcal O_{\overline S}(m)|$. 
By Bertini's theorem (cf.~\cite[Chap.~II, Theorem~8.18]{Hartshorne}), $C$ is smooth.
Moreover, since $H^1(\mathcal O_S(-mD))=0$ for sufficiently large $m$
by Lemma~\ref{lem:vanishing-D}, it follows from the exact sequence
$0 \rightarrow \mathcal O_S(-C)
\rightarrow \mathcal O_S
\rightarrow \mathcal O_C
\rightarrow 0$ that $C$ is connected.
Thus \eqref{item:nonsingularity and dimension of flag} 
follows from Lemma~\ref{lem:smoothness of flag3}.
Secondly, \eqref{item:conclusion for n>3} follows from Proposition~\ref{prop:generic smoothness}.
Finally we prove \eqref{item:conclusion for n=3}. 
Suppose that $n=3$ and let $W_{C,S}$ be the $S$-maximal family in 
$\Hilb^{sc} \mathbb P^5$ containing $[C]$.
Since we have $H^0(N_{S/\mathbb P^5}(-C))=0$ by Lemma~\ref{lem:exterior deformations},
the first projection $pr_1$ embeds the tangent space of 
$\HF^{sc} \mathbb P^5$ at $(C,S)$ into that of $\Hilb^{sc} \mathbb P^5$ at $[C]$.
This implies $\dim W_{C,S}=\dim_{(C,S)} \HF^{sc} \mathbb P^5$.
Then the same lemma shows that
$$
\begin{aligned}
  h^0(C,N_{C/\mathbb P^5})
  &=h^0(\mathbb P^5,N_{(C,S)/\mathbb P^5})+h^1(S,N_{S/\mathbb P^5}(-C))\\
  &=\dim W_{C,S}+3.
\end{aligned}
$$
Then by virtue of Proposition~\ref{prop:obstruction},
we have $\ob(\alpha)\ne 0$ for every $\alpha \not\in \im p_1$,
where $p_1$ is the tangent map of 
the first projection 
$pr_1: \HF^{sc} \mathbb P^5 \rightarrow \Hilb^{sc} \mathbb P^5$ at $(C,S)$.
This implies that $\Hilb^{sc} \mathbb P^5$ is singular at $[C]$
and moreover by Lemma~\ref{lem:maximality}, 
$W_{C,S}$ is an irreducible component of $(\Hilb^{sc} \mathbb P^5)_{\red}$
and 
$$
\dim_{(C,S)} \HF^{sc} \mathbb P^5 = \dim_{[C]} \Hilb^{sc} \mathbb P^5.
$$
By replacing $(C,S)$ with a general member
$(C',S')$ of the unique irreducible 
component $\mathcal W_{C,S}$ of $\HF^{sc} \mathbb P^5$
passing through $(C,S)$, we have
$h^0(C',N_{C'/\mathbb P^5})=\dim_{[C']} \Hilb^{sc} \mathbb P^5+3$ 
and $\Hilb^{sc} \mathbb P^5$ is singular at $[C']$.
Thus we have completed the proof of Theorem~\ref{thm:main1}.
\qed

\begin{rmk}
  The degree $d$ and genus $g$ of curves $C$ in Theorem~\ref{thm:main1} are given by
  $$
  \begin{aligned}
    d&=m(2^n (n-1)+1),\\
    g&=\dfrac{m(2^n(n-1)+1)(m(n-1)+n-3)}2+1
  \end{aligned}
  $$
  by $d=C\cdot \mathbf h$, $\mathbf h^2=2^n$,
  $K_S=(n-3)\mathbf h$ and adjunction formula. The formula \eqref{eqn:formula for dim|mD|}
  for $\lambda_{n,m}$ ($=\chi(S,\mathcal O_S(C))-1$) can be recovered from
  $$
  \begin{aligned}
    \chi(S,\mathcal O_S(C)) &=\dfrac12 (C^2+C\cdot K_S)-C\cdot K_S+\chi(S,\mathcal O_S)\\
    &=g-1-(n-3)d+\chi(S,\mathcal O_S)
  \end{aligned}
  $$
  by substituting the above formulas for $d$ and $g$, together with \eqref{eqn:euler char}.
\end{rmk}

\begin{rmk}
  \label{rmk:n=2 and n=1}
  In the cases $n=1$ and $n=2$, 
  we still have a similar statement to Theorem~\ref{thm:main1}~(1), but
  the statements (2) and (3) never hold.
  Indeed, $C$ is a smooth curve and 
  $\HF \mathbb P$ is nonsingular at $(C,S)$ of dimension $m+9$ if $n=1$ and 
  $5m(m+1)/2+26$ if $n=2$.
  See e.g.~\cite[Lemma~13]{Kleppe87} and \cite[Lemma~2.3]{Nasu10} for the proof.
  If $n=1$ then $C$ is not connected.
  In fact, $S_2$ is a smooth quadric surface ($\simeq \mathbb P^1 \times \mathbb P^1$)
  in $\mathbb P^3$, and since $C\sim mE$, 
  $C$ is a disjoint union of $m$ rulings of $S_2\to \mathbb P^1$.
  Moreover, it is known that
  every irreducible component of $\Hilb^{sc} \mathbb P^3$ 
  is of dimension at least $4d=4m$. 
  Thus $C$ is not a general member of any component if $n=1$ and $4m> m+9$.
  If $n=2$, then every general curve $C\sim m(\mathbf h+E) \subset S_{2,2}$ is connected, but
  the $S$-maximal family $W_{C,S}$ containing $[C]$ is not an irreducible component 
  of $\Hilb^{sc} \mathbb P^4$.
  Indeed, it follows from \cite[Example~3.8]{Nasu10} that
  $C$ is a projection of a curve $\Gamma \subset \mathbb P^5$,
  sitting on a quintic del Pezzo surface $F \subset \mathbb P^5$,
  which is an unprojection morphism of $E$ on $S=S_{2,2}$.
  In particular, $C$ is not linearly normal.
  Moreover, we see that if $m\ge 3$ then there exists a global deformation of 
  $C \subset \mathbb P^4$ such that $C \not\subset S'_{2,2}$ 
  for any $S'_{2,2} \subset \mathbb P^4$ by 
  deformations of the center of the (un)projection.
  Therefore, $W_{C,S}$ is not a component of $\Hilb^{sc} \mathbb P^4$ if $m\ge 3$.
\end{rmk}

\paragraph{\bf Acknowledgment}

The author thanks Professor Euisung Park 
for inviting him to give a series of lectures on deformation theory 
at Korea University in September 2026. 
The preparation of these lectures helped him clarify 
some of the ideas presented in this paper.
The author acknowledges the use of ChatGPT (OpenAI) 
during the preparation of this manuscript. 
In particular, discussions with ChatGPT provided useful ideas 
for the proofs of Lemmas~\ref{lem:higher-direct-image} 
and \ref{lem:vanishing-D}. 
All mathematical arguments and statements appearing 
in the final manuscript were independently checked and 
verified by the author, who takes full responsibility 
for the content of the paper. 
This work was supported in part by JSPS KAKENHI Grant Number JP24K06677.

\bibliography{mybib}

\begin{thebibliography}{10}

\bibitem{CIK24a}
Y.~Choi, H.~Iliev, and S.~Kim.
\newblock Components of the {H}ilbert scheme of smooth projective curves using
  ruled surfaces {II}: existence of non-reduced components.
\newblock {\em Manuscripta Math.}, 175(3-4):921--942, 2024.

\bibitem{CIK24b}
Y.~Choi, H.~Iliev, and S.~Kim.
\newblock Non-reduced components of the {H}ilbert scheme of curves using triple
  covers.
\newblock {\em Mediterr. J. Math.}, 21(4):Paper No. 134, 20, 2024.

\bibitem{CLM96}
C.~Ciliberto, A.~F. Lopez, and R.~Miranda.
\newblock Some remarks on the obstructedness of cones over curves of low genus.
\newblock In {\em Higher-dimensional complex varieties ({T}rento, 1994)}, pages
  167--182. de Gruyter, Berlin, 1996.

\bibitem{Dan17}
A.~Dan.
\newblock On generically non-reduced components of {H}ilbert schemes of smooth
  curves.
\newblock {\em Math. Nachr.}, 290(17-18):2800--2814, 2017.

\bibitem{Ellia87}
P.~Ellia.
\newblock D'autres composantes non r\'eduites de {${\rm Hilb}\,{\bf P}\sp 3$}.
\newblock {\em Math. Ann.}, 277(3):433--446, 1987.

\bibitem{Ellingsrud-Peskine}
G.~Ellingsrud and C.~Peskine.
\newblock Anneau de {G}orenstein associ\'{e} \`a un fibr\'{e} inversible sur
  une surface de l'espace et lieu de {N}oether-{L}efschetz.
\newblock In {\em Proceedings of the {I}ndo-{F}rench {C}onference on {G}eometry
  ({B}ombay, 1989)}, pages 29--42. Hindustan Book Agency, Delhi, 1993.

\bibitem{Floystad93}
G.~Fl{\o}ystad.
\newblock Determining obstructions for space curves, with applications to
  nonreduced components of the {H}ilbert scheme.
\newblock {\em J. Reine Angew. Math.}, 439:11--44, 1993.

\bibitem{TDTEGA5}
A.~Grothendieck.
\newblock Technique de descente et th\'eor\`emes d'existence en g\'eom\'etrie
  alg\'ebrique. {V}. {L}es sch\'emas de {P}icard: th\'eor\`emes d'existence.
\newblock In {\em S\'eminaire {B}ourbaki, {V}ol.\ 7}, pages Exp.\ No.\ 232,
  143--161. Soc. Math. France, Paris, 1995.

\bibitem{Gruson-Peskine82}
L.~Gruson and C.~Peskine.
\newblock Genre des courbes de l'espace projectif. {II}.
\newblock {\em Ann. Sci. \'Ecole Norm. Sup. (4)}, 15(3):401--418, 1982.

\bibitem{Hartshorne}
R.~Hartshorne.
\newblock {\em Algebraic geometry}.
\newblock Springer-Verlag, New York, 1977.
\newblock Graduate Texts in Mathematics, No. 52.

\bibitem{Hartshorne10}
R.~Hartshorne.
\newblock {\em Deformation theory}, volume 257 of {\em Graduate Texts in
  Mathematics}.
\newblock Springer, New York, 2010.

\bibitem{Kleppe81}
J.~O. Kleppe.
\newblock The {H}ilbert-flag scheme, its properties and its connection with the
  {H}ilbert scheme. {A}pplications to curves in 3-space.
\newblock Preprint (part of thesis), Univ. of Oslo, 1981.

\bibitem{Kleppe87}
J.~O. Kleppe.
\newblock Nonreduced components of the {H}ilbert scheme of smooth space curves.
\newblock In {\em Space curves (Rocca di Papa, 1985)}, volume 1266 of {\em
  Lecture Notes in Math.}, pages 181--207. Springer, Berlin, 1987.

\bibitem{Kleppe06}
J.~O. Kleppe.
\newblock The {H}ilbert scheme of space curves of small diameter.
\newblock {\em Ann. Inst. Fourier (Grenoble)}, 56(5):1297--1335, 2006.

\bibitem{Kleppe-Ottem15}
J.~O. Kleppe and J.~C. Ottem.
\newblock Components of the {H}ilbert scheme of space curves on low-degree
  smooth surfaces.
\newblock {\em Internat. J. Math.}, 26(2):1550017, 30, 2015.

\bibitem{Mukai-Nasu}
S.~Mukai and H.~Nasu.
\newblock Obstructions to deforming curves on a 3-fold, {I}: {A} generalization
  of {M}umford's example and an application to {H}om schemes.
\newblock {\em J. Algebraic Geom.}, 18(4):691--709, 2009.

\bibitem{Mumford}
D.~Mumford.
\newblock Further pathologies in algebraic geometry.
\newblock {\em Amer. J. Math.}, 84:642--648, 1962.

\bibitem{Nasu1}
H.~Nasu.
\newblock Obstructions to deforming space curves and non-reduced components of
  the {H}ilbert scheme.
\newblock {\em Publ. Res. Inst. Math. Sci.}, 42(1):117--141, 2006.

\bibitem{Nasu4}
H.~Nasu.
\newblock Obstructions to deforming curves on a 3-fold, {II}: {D}eformations of
  degenerate curves on a del {P}ezzo 3-fold.
\newblock {\em Ann. Inst. Fourier (Grenoble)}, 60(4):1289--1316, 2010.

\bibitem{Nasu5}
H.~Nasu.
\newblock Obstructions to deforming curves on a 3-fold, {III}: {D}eformations
  of curves lying on a {$K3$} surface.
\newblock {\em Internat. J. Math.}, 28(13):1750099, 30, 2017.

\bibitem{Nasu6}
H.~Nasu.
\newblock Obstructions to deforming curves on a prime {F}ano 3-fold.
\newblock {\em Math. Nachr.}, 292(8):1777--1790, 2019.

\bibitem{Nasu7}
H.~Nasu.
\newblock Corrigendum to ``{O}bstructions to deforming curves on a 3-fold,
  {III}: {D}eformations of curves lying on a {$K3$} surface''.
\newblock {\em Internat. J. Math.}, 31(12):2092001, 6, 2020.

\bibitem{Nasu8}
H.~Nasu.
\newblock Obstructions to deforming curves on an {E}nriques-{F}ano $3$-fold.
\newblock {\em J. Pure Appl. Algebra}, 225(9):106677, 15, 2021.

\bibitem{Nasu9}
H.~Nasu.
\newblock Obstructions to deforming space curves lying on a smooth cubic
  surface.
\newblock {\em Manuscripta Math.}, 172(1-2):31--55, 2023.

\bibitem{Nasu10}
H.~Nasu.
\newblock Obstructions to deforming space curves lying on a del {P}ezzo
  surface.
\newblock {\em Manuscripta Math.}, 177(3):Paper No. 33, 30, 2026.

\bibitem{Papadakis-Reid04}
S.~A. Papadakis and M.~Reid.
\newblock Kustin-{M}iller unprojection without complexes.
\newblock {\em J. Algebraic Geom.}, 13(3):563--577, 2004.

\bibitem{Perrin08}
D.~Perrin.
\newblock {\em Algebraic geometry}.
\newblock Universitext. Springer-Verlag London, Ltd., London; EDP Sciences, Les
  Ulis, 2008.
\newblock An introduction, Translated from the 1995 French original by Catriona
  Maclean.

\bibitem{Reid00}
M.~Reid.
\newblock Graded rings and birational geometry.
\newblock In K.~Ohno, editor, {\em Algebraic Geometry Symposium}, Departmental
  Bulletin Paper, pages 1--72, 2000.

\bibitem{Sernesi}
E.~Sernesi.
\newblock {\em Deformations of algebraic schemes}, volume 334 of {\em
  Grundlehren der Mathematischen Wissenschaften [Fundamental Principles of
  Mathematical Sciences]}.
\newblock Springer-Verlag, Berlin, 2006.

\end{thebibliography}
\bibliographystyle{abbrv}

\end{document}